\documentclass[11pt]{amsart}
\usepackage{etex}
\usepackage{amsmath, amscd, amsthm, amssymb, graphics, mathrsfs, setspace,fancyhdr,geometry,tabularx,shapepar, xcolor, tikz, cite, booktabs, dsfont,amsfonts, marvosym, amsbsy, bm, tikz, array, mathdots, dsfont, subcaption}
\usetikzlibrary{matrix,arrows,backgrounds}
\usetikzlibrary{decorations.pathmorphing,shapes}
\usepackage[hidelinks]{hyperref}
\hypersetup{colorlinks=false}
\theoremstyle{plain}

\DeclareCaptionSubType*[alph]{figure}

\usetikzlibrary{decorations.pathreplacing}

\usepackage[pagewise]{lineno}

\newtheorem{thm}{Theorem}[section]

\newtheorem{lem}[thm]{Lemma}
\newtheorem{cor}[thm]{Corollary}
\newtheorem{prop}[thm]{Proposition}

\newtheorem{question}[thm]{Question}

\newtheorem*{thm*}{Theorem}
\newtheorem*{problem*}{Problem}

\theoremstyle{definition}
\newtheorem{defn}[thm]{Definition}

\newtheorem{notation}[thm]{Notation}
\newtheorem{rem}[thm]{Remark}

\newtheorem{ex}[thm]{Example}

\numberwithin{equation}{section}

\newcommand{\Spec}{\operatorname{Spec}}

\renewcommand{\O}{{\mathcal O}}
\renewcommand{\hom}{\operatorname{Hom}}
\newcommand{\real}{{\mathbb R}}
\newcommand{\cplx}{{\mathbb C}}

\newcommand{\Z}{{\mathbb Z}}

\newcommand{\id}{\operatorname{id}}

\newcommand{\gm}{\mathbb{G}_{m}}

\newcommand{\Ext}{\operatorname{Ext}}

\newcommand{\Aut}{\operatorname{Aut}}
\newcommand{\Pic}{\operatorname{Pic}}

\newcommand{\bbP}{\mathbb{P}}
\newcommand{\bbA}{\mathbb{A}}

\newcommand{\modmod}[1]{/ \! \! / \!_{#1}}

\newcommand{\leftexp}[2]{{\vphantom{#2}}^{#1}{#2}}
\newcommand{\weezer}{\leftexp{=}{\kern-0.23em\mathsf{W}}^{\kern-0.21em =}}
\newcommand{\oldF}{\mathsf{G}}

\begin{document}

\title[Derived categories of centrally-symmetric toric Fanos]{Derived categories of centrally-symmetric smooth toric Fano varieties}

{\small
\author[Ballard]{Matthew R Ballard}
\address{Department of Mathematics, University of South Carolina, 
Columbia, SC 29208}
\email{ballard@math.sc.edu}
\urladdr{\url{http://www.matthewrobertballard.com}}
\thanks{}

\author[Duncan]{Alexander Duncan}
\address{Department of Mathematics, University of South Carolina, 
Columbia, SC 29208}
\email{duncan@math.sc.edu}
\urladdr{\url{http://people.math.sc.edu/duncan/}}
\thanks{}

\author[McFaddin]{Patrick K. McFaddin}
\address{Department of Mathematics, Fordham University, 113 W 60th St., New York, NY 10023}
\email{pmcfaddin@fordham.edu}
\urladdr{\url{http://mcfaddin.github.io/}}
\thanks{}
}

\begin{abstract}
 We exhibit full exceptional collections of vector bundles on any smooth, Fano arithmetic toric variety whose split fan is centrally symmetric. 
\end{abstract}

\maketitle
\addtocounter{section}{0}



\section{Introduction}

In recent years, there has been an explosion of work focused on derived categories and their connections to the geometry of algebraic varieties. Such work has brought applications to birational geometry of varieties over $\cplx$ \cite{AT, ABB, BB, BMMSS, KuznetsovCubic4fold, Vial} and $k$-linear triangulated categories for general fields $k$ \cite{AKW, AAGZ, ADPZ, HT, Honigs, LiebMaulSnow}. Naturally, derived categories have also proved to be an invaluable invariant for studying birational
geometry over arbitrary fields \cite{AB} as well as for solving problems concerning algebraic $K$-theory and additive invariants in general \cite{Tabuada}.

Given a $k$-variety $X$, one gains insight into the structure of the coherent derived category $\mathsf{D^b}(X) = \mathsf{D^b}( \text{coh} X)$ by treating it as a ``vector space" and exhibiting a semi-orthonormal ``basis" relative to the ``bilinear form" $\hom _{\mathsf{D^b}(X)}(-, -)$.  In the proper terminology, one exhibits an
\emph{exceptional collection} consisting of \emph{exceptional objects}.
Originally put forth by Beilinson in \cite{Beilinson}, an exceptional object of the $k$-linear triangulated category $\mathsf{D^b}(X)$ is one whose endomorphism algebra is isomorphic to the base field $k$. In the case when $k$ is not algebraically closed, this definition is much too restrictive. This formalism does not allow for exceptional objects whose endomorphism algebra is a non-trivial simple $k$-algebra, e.g., a field extension of $k$.  The existence of such algebras reflects the arithmetic complexity of $k$ via the associated classes in the Brauer group.

A more natural definition: an object of $\mathsf{D^b}(X)$ is \emph{exceptional} if its endomorphism algebra is a division $k$-algebra (concentrated in homological degree zero) \cite{AB, BDM}. An exceptional collection is then given by a totally ordered set $\mathsf{E} = \{E_1, ..., E_s\}$ of exceptional objects in $\mathsf{D^b}(X)$ satisfying $\text{Ext}^n(E_i, E_j) = 0$ for all integers $n$ whenever $i >j$.  An exceptional collection is $\emph{full}$ if it generates $\mathsf{D^b}(X)$, i.e., the smallest thick subcategory of $\mathsf{D^b}(X)$ containing $\mathsf{E}$ is all of $\mathsf{D^b}(X)$ (see Definition~\ref{def:exceptional} for details).

Exceptional collections provide the most atomic decomposition of the
derived category of coherent sheaves on a variety. They have rich ties
to representation theory of finite-dimensional algebras and their
existence has strong structural implications for the motive of a variety, both in the commutative and non-commutative settings \cite{OrlovChow, TabuadaChow}.
However, the following question is still very open:
\begin{question} \label{ques:excoll}
 Which smooth projective varieties admit full exceptional collections? 
\end{question}
In particular, even in cases where one knows that the answer to this
question is positive, techniques for constructing full exceptional
collections can be highly idiosyncratic. 

Toric varieties defined over algebraically-closed fields of
characteristic zero provide an important testing ground which informs
our understanding of the existence and construction of exceptional
collections. The existence of such collections was settled affirmatively in \cite{Kawamata, Kawamata2} (see also \cite{BFK}). Moving beyond to general fields and twisted forms of toric varieties (also called \emph{arithmetic toric
varieties} \cite{Duncan, ELFST, MerkPan}), one has the opportunity to
further advance our grasp of the general situation. Indeed, this
presents a non-trivial challenge: it is not known whether all smooth
projective arithmetic toric varieties admit full exceptional
collections. On the other hand, base changing a field $k$ to its
separable closure $k^{\text{sep}}$ opens the door to many useful tools
and known results.

In \cite{BDM}, the authors showed that a smooth projective $k$-variety $X$ (not required
to be toric) admits a full exceptional collection if and only if
$X_{k^{\text{sep}}}$ admits a full exceptional collection which is
Galois-stable, i.e., objects of the collection are permuted by the
action of $\text{Gal}(k^{\text{sep}}/k)$. Exhibiting a full exceptional
collection over $k$ therefore requires that one produce a collection
over $k^{\text{sep}}$ which is highly symmetric with respect to the
Galois action. By considering the class of toric varieties, one quickly
recognizes that ``most'' full exceptional collections are not
Galois-stable. The Galois-stable collections are often the simplest,
particularly due to their large exceptional blocks (subcollections
consisting of objects which are mutually orthogonal). One may
optimistically hope that the additional constraint of Galois-stability
makes the search for a positive answer to Question~\ref{ques:excoll}
more tractable in general. 

In this paper, we study a particular highly-symmetric class of smooth projective varieties. A polytope $P \subseteq \real ^m$ is \emph{centrally symmetric} if it satisfies $-P = P$. The smooth split toric varieties $X$ whose anti-canonical polytope is full-dimensional and centrally symmetric were classified in \cite{VosKly}. It was shown in loc. cit. that any such variety, which we refer to as \emph{centrally symmetric toric Fano varieties}, is isomorphic to a product of projective lines and \emph{generalized del Pezzo varieties} $V_n$ of dimension $n = 2m$.

The variety $V_n$ is the (split) toric variety with rays given by
$$\begin{array}{rl}
e_0 &= (-1,-1,\cdots,-1)\\
e_1 &= (1,0,\cdots,0)\\
e_2 &= (0,1,\cdots,0)\\
&\vdots\\
e_n &= (0,0,\cdots,1)
\end{array}
\hspace{.4cm}
\begin{array}{rl}
\bar{e}_0 &= (1,1,\cdots,1)\\
\bar{e}_1 &= (-1,0,\cdots,0)\\
\bar{e}_2 &= (0,-1,\cdots,0)\\
&\vdots\\
\bar{e}_n &= (0,0,\cdots,-1)
\end{array}$$ and whose maximal cones are as follows
(see \cite[Proof of Thm. 5]{VosKly}).
Each maximal cone is generated by the rays in the set
$\{ e_i \}_{i \in A} \cup \{ \bar{e}_i \}_{i \in B}$
where $A$ and $B$ are disjoint subsets of $\{0, \ldots, n\}$, each
of cardinality $\frac{n}{2}$.
The number of maximal cones $c(n)$ of $V_n$ is 
\[
c(n)=\frac{(n+1)!}{(\frac n 2)!^2} = \frac{(2m +1)! }{m!^2}\ , 
\]
which coincides with the rank of Grothendieck group $K_0(V_n)$.
Throughout, we let $\Delta$ denote the fan corresponding to $V_n$.
Note that $V_2$ is the del Pezzo surface $\mathsf{dP}_6$ of degree 6;
and $V_4$ is the variety (116) in the enumeration of \cite{Prabhu} or
(118) in the enumeration of \cite{Batyrev}.
Any odd-dimensional centrally symmetric toric Fano variety has $\bbP^1$
as a factor and there are no generalized del Pezzo surfaces of odd
degree.

The variety $V_n$ admits a natural $(S_{n+1} \times C_2)$-action, given
by an action on the rays $e_i, \overline{e}_i$.  The $S_{n+1}$-action
permutes $e_0,\ldots,e_n$ and $\bar{e}_0\ldots\bar{e}_0$ in the obvious
way.  The $C_2$-action, whose generator we refer to as \emph{the
antipodal involution}, is the antipodal map on the cocharacter lattice,
and interchanges $e_i$ and $\bar{e}_i$ for each index $i$.

The variety $V_n$ is of importance in birational geometry due to its
appearance in the factorization of the standard Cremona transformation
of $\bbP^n$, and may be constructed in an entirely geometric manner. First, take the blow-up of $\mathbb{P}^n$ at the collection of $(n+1)$ torus fixed points, then
flip the (strict transforms) of the lines through these points, then
flip the (strict transforms) of planes through these points, and so on, up
to, but not including, the half-dimensional linear subspaces. The
resulting variety is $V_n$ \cite[\S3]{Casagrande}. 

\begin{notation}
We let $Y_n$ denote the blowup of $\mathbb{P}^n$ at its ($n+1$) torus fixed points.
\end{notation}

 Since $V_n$ and $Y_n$  are
isomorphic in codimension $1$, they have isomorphic Picard groups. We let $H,E_0, \ldots, E_n$ be a basis for $\Pic(V_n)$,
given by the hyperplane and 
exceptional divisors of $Y_n$.
The divisors corresponding to the rays $e_i, \overline{e}_i$ are then given by
\[
[e_i] = E_i, \quad
[\bar{e}_i] = (H-\sum_{j=0}^{n}E_j)+E_i,
\]
where $S_{n+1}$ permutes the $E_i$ leaving $H$ fixed, and the antipodal involution
is represented by the matrix
\[
\left( \begin{array}{rrrrr}
n & 1 & 1 & \cdots & 1 \\
1-n & 0 & -1 & \cdots & -1 \\
1-n & -1 & 0 & \cdots & -1 \\
\vdots & \vdots & \vdots & \ddots & \vdots\\
1-n & -1 & -1 & \cdots & 0 \\
\end{array} \right).
\]
For each $c \in \Z$ and $J \subseteq \{0, \ldots, n\}$, define 
\[
 F_{c,J} := c\left(\sum_{i=0}^{n}E_i - H\right) - \sum_{j \in J} E_j .
\]
Note that the antipodal involution takes $F_{c,J}$ to $F_{|J|-c,J}$. 

\begin{thm} \label{thm:Vn}
 The set $\oldF_n$ consisting of all line bundles of the form $\O(F_{c,J})$ with
 \begin{enumerate}
  \item $\displaystyle{|J| \le \frac{n}{2},\ |J|-\frac{n}{4} \le c \le \frac{n}{4}}$, or
  \item $\displaystyle{|J| \ge \frac{n}{2}+1,\ \frac{n+2}{4} \le c \le |J| - \frac{n+2}{4}}$
 \end{enumerate}
forms a full strong $(S_{n+1} \times C_2)$-stable exceptional collection of line bundles
on $V_n$ under any ordering of the
blocks such that $|J|$ is (non-strictly) decreasing.
\end{thm}

Note that, in general, smooth projective toric Fano varieties do not
have full strong exceptional collections of line bundles \cite{Efimov},
let alone a highly symmetric such collection.

Recall that a \emph{form} of a $k$-variety $X$ is a $k$-variety $X'$
such that there is an isomorphism $X_K \cong X'_K$ after base change
to some field extension $K/k$.
From \cite[Proposition 3.7]{BDM}, we see that if a (split) toric variety has an
exceptional collection that is stable under the full automorphism of its
fan, then we obtain exceptional collections for its twisted forms.
Since $\Aut(V_n) \cong S_{n+1} \rtimes C_2$, this applies to the
varieties $V_n$.
In fact, \cite[Lemma 3.11]{BDM} allows us to extend this to products of
such varieties.
Since any centrally symmetric toric Fano variety is a product of
projective lines and the varieties $V_n$, the aforementioned results
yield the following:

\begin{cor}\label{cor:centsym}
Any form of a centrally symmetric toric Fano variety admits a full
strong exceptional collection consisting of vector bundles.
\end{cor}

In \cite[Theorem 6.6]{CT}, Castravet and Tevelev exhibit a full strong $\Aut(\Delta)$-stable exceptional collection for $V_n$, where $\Delta$ denotes the fan associated to $V_n$.
The authors of this paper had independently discovered the same collection (up to a twist by a line bundle), as discussed in \cite[\S4.4]{BDM}. This article fleshes out these ideas. A particular benefit is that complexity in checking generation in \cite{CT} is lessened greatly by the methods here. In particular, we make use of forbidden cones in showing exceptionality and grade-restriction windows in showing fullness (i.e., that the collection generates the bounded derived category $\mathsf{D^b}(V_n)$). The distinct methods and perspective should be valuable in understanding 
more general situations. 

\subsection*{Acknowledgements}
\label{sec:hat_tip}

Via the first author, this material is based upon work supported by the National Science Foundation under Grant No.~NSF DMS-1501813. Via the second author, this work was supported by a grant from the Simons Foundation/SFARI (638961, AD) and by NSA grant H98230-16-1-0309. The third author was partially supported by an AMS-Simons travel grant. The authors would like to thank the anonymous referee for useful comments and a thorough reading of an earlier draft.

\section{Examples}

Let us explicitly describe the full exceptional collections in low-dimensional examples. We remind the reader that we utilize the aforementioned basis of $\text{Pic}(Y_n)$, and we also set $E: = \sum E_i$.

\begin{ex}[Dimension 2]
Applying the inequalities given in Theorem~\ref{thm:Vn},
we see that the pairs $(c,|J|)$ that occur are
$F_2 = \{(0, 0), (1, 2), (1,3), (2, 3)\}$.
Each of these pairs gives an $S_3$-orbit of bundles on $V_2$.

\begin{center}
\begin{tabular}{| c | l |}
\hline
$(c,|J|)$ & $S_3$-orbit \\
\hline
$(2, 3)$ & $\O(E - 2H)$ \\
\hline
$(1, 3)$ & $\O(-H)$ \\
\hline
$(1, 2)$ & $\O( E_1-H), \O(E_2 -H), \O(E_3 -H)$ \\
\hline
$(0, 0)$ & $\O$ \\
\hline
\end{tabular}
\end{center}
\end{ex}

\noindent Notice that this is the exceptional collection on $V_2 =
\mathsf{dP}_6$ (the del Pezzo surface of degree 6) which is the dual  of
that given in \cite[Prop. 6.2)(ii)]{King}. This collection was also recovered in \cite{BSS}. Recall that the antipodal involution acts on these
orbits via $(c, \ell) \mapsto (\ell -c , \ell)$, so that $(1, 3) \mapsto
(2, 3)$.  We thus obtain (orthogonal) blocks $\mathsf{E}_i$ given by the $(S_{3} \times
C_2)$-orbits:
\[
\begin{tabular}{| r | l |}
\hline
 $ \mathsf{E}_2 $& $ \O(-H), \O(E-2H)$\\
\hline
$ \mathsf{E}_1$ & $\O( E_1-H), \O(E_2 -H), \O(E_3 -H)$ \\
\hline
$\mathsf{E}_0$ & $ \O$ \\
\hline
\end{tabular}
\]

\begin{ex}[Dimension 4]
The variety $V_4$ is exactly (116) in the enumeration of \cite{Prabhu} or
(118) in the enumeration of \cite{Batyrev}.
Applying the inequalities given in Theorem~\ref{thm:Vn},
we see that the possible $(c,|J|)$ make up the set
\[
F_4 :=
\{(-1,0), (0,0), (1,0), (0, 1), (1, 1), (1, 2), (2, 4), (2, 5), (3, 5)\}.
\]
Each of these pairs gives an $S_5$-orbit of bundles on $V_4$. 

\begin{center}
\begin{tabular}{| c | l |}
\hline
$(c,|J|)$ & $S_5$-orbit \\
\hline
$(3, 5)$ & $\O(2E - 3H) $ \\
\hline
$(2, 5)$ & $\O(E-2H) $ \\
\hline
$(2, 4)$ & $\O(E -2H + E_1), \O(E -2H + E_2), \O(E -2H + E_3), \O(E -2H + E_4), \O(E -2H +  E_5)$\\
\hline
$(1,2)$ & $\O(E - H - E_1 -E_2), \O(E - H - E_1 -E_3), \O(E - H - E_1 -E_4), \O(E - H - E_1 -E_5),$\\
&  $\O(E - H - E_2 -E_3),\O(E - H - E_2 -E_4), \O(E - H - E_2 -E_5),$ \\
& $\O(E - H - E_3 -E_4), \O(E - H - E_3 -E_5), \O(E - H - E_4 -E_5)$\\
\hline
$(1, 1)$ & $\O(E - H -E_1), \O(E - H - E_2), \O(E - H - E_3), \O(E - H - E_4), \O(E - H - E_5)$ \\
\hline
$(1,0)$ & $\O(E -H) $ \\
\hline
$(0, 1)$ &$ \O(-E_1), \O(-E_2), \O(-E_3), \O(-E_4), \O(-E_5) $ \\
\hline
$(0,0)$ & $\O$ \\
\hline
$(-1, 0)$ &$ \O(-E + H)$ \\
\hline
\end{tabular}
\end{center}

The antipodal involution acts on these orbits via $(-1, 0)
\mapsto (1, 0) $, $(0, 1) \mapsto (1,1)$, and $(2, 5)
\mapsto (3, 5),$ leaving the others fixed. We thus obtain
(orthogonal) blocks $\mathsf{E}_i$ given by the $(S_5\times
C_2)$-orbits:

\begin{center}
\begin{tabular}{| r | l |}
\hline
$\mathsf{E}_5 $&  $\O(E-2H), \O(2E -3H) $\\
\hline
$\mathsf{E}_4$ & $ \O(E -2H + E_1), \O(E -2H + E_2), \O(E -2H + E_3),$ \\
&$\O(E -2H + E_4), \O(E -2H +  E_5) $ \\
\hline
$\mathsf{E}_3$ & $\O(E - H - E_1 -E_2), \O(E - H - E_1 -E_3), \O(E - H - E_1 -E_4),$\\
&  $ \O(E - H - E_1 -E_5), \O(E - H - E_2 -E_3),\O(E - H - E_2 -E_4),$ \\
& $\O(E - H - E_2 -E_5) \O(E - H - E_3 -E_4), \O(E - H - E_3 -E_5),$\\
&  $ \O(E - H - E_4 -E_5) $\\
\hline
$\mathsf{E}_2$ & $ \O(-E_1), \O(-E_2), \O(-E_3), \O(-E_4), \O(-E_5), \O(E - H -E_1)$ \\
& $ \O(E - H - E_2), \O(E - H - E_3), \O(E - H - E_4), \O(E - H - E_5)$\\
\hline
$\mathsf{E}_1$ & $\O(-E + H), \O(E -H)$ \\
\hline
$\mathsf{E}_0$ &  $\O$ \\
\hline
\end{tabular}
\end{center}
\end{ex}

\section{Exceptionality via forbidden cones}

We begin by recalling definitions of exceptional objects and collections. We then apply the theory of forbidden cones, put forth by Borisov and Hua \cite{BH}, to show that the collection described above is exceptional and stable under the action of the group $S_{n+1} \times C_2$. For a $k$-scheme $X$, we let $\mathsf{D^b}(X) = \mathsf{D^b}(\text{coh} X)$ denote the bounded derived category of coherent sheaves on $X$. It is a $k$-linear triangulated category.

\begin{defn}\label{def:exceptional}
Let $\mathsf{T}$ be a $k$-linear triangulated category.  An object $E$ in $\mathsf{T}$ is \emph{exceptional} if the following conditions hold:
\begin{enumerate}
\item $\text{End}_{\mathsf{T}}(E)$ is a division $k$-algebra.
\item $\text{Ext}^n_{\mathsf{T}}(E, E) := \text{Hom}_{\mathsf{T}}(E, E[n]) = 0$ for $n \neq 0$.
\end{enumerate}
A totally ordered set $\mathsf{E} = \{E_1, ..., E_s\}$ of exceptional objects in $\mathsf{T}$ is an \emph{exceptional collection} if $\text{Ext}^n_{\mathsf{T}}(E_i, E_j)   =  0$ for all integers $n$ whenever $i >j$.  An exceptional collection is $\emph{full}$ if it generates $\mathsf{T}$, i.e., the smallest thick subcategory of $\mathsf{T}$ containing $\mathsf{E}$ is all of $\mathsf{D^b}(X)$.  An exceptional collection is \emph{strong} if $\text{Ext}^n_{\mathsf{T}}(E_i, E_j) = 0$ whenever $n \neq 0$.  An \emph{exceptional block} is an exceptional collection $\mathsf{E} = \{ E_1, ..., E_s\}$ such that $\text{Ext}^n_{\mathsf{T}}(E_i, E_j) = 0$ for every $n$ whenever $i \neq j$.
\end{defn}

\begin{defn}
Let $X$ be a scheme with an action of a group $G$.  Any element $g
\in G$ induces a functor $g^*: \mathsf{D^b}(X) \to \mathsf{D^b}(X)$. A $G$-\emph{stable
exceptional collection} on $X$ is an exceptional collection $\mathsf{E}
= \{E_1, ..., E_s\}$ of objects in $\mathsf{D^b}(X)$ such that
for all $g \in G$ and $1 \leq i \leq s$ there exists $E \in \mathsf{E}$
such that $g^*E_i \simeq E$.
\end{defn}

Let us now investigate exceptionality of the collection $\oldF_n$ described in Theorem~\ref{thm:Vn}. It will be useful for our calculations to consider a larger collection of bundles grouped into $S_{n+1}$-orbits. Conceptually, this gives a nice picture of (orbits of objects in) this collection as shown in Figures \ref{fig:graph2}, \ref{fig:graph4}, \ref{fig:graph6}, and \ref{fig:graph8}.
Suppose $\ell,k$ are nonnegative integers such that $k+\ell \le n+1$.
Let $F(c,k,\ell)$ be the $S_{n+1}$-orbit of the divisor
\[
c\left(\sum_{i=0}^{n}E_i - H\right) + \sum_{i=0}^{k-1}E_i
- \sum_{i=k}^{k+\ell-1}E_i \ .
\]
Notice that when $k=0$, the set $F(c, k, \ell)$ is just the
$S_{n+1}$-orbit of the divisor $F_{c, J}$, where $J$ is any subset of $\{0, 1, \ldots, n\}$ of cardinality $\ell$. Also note that $F(c,k,\ell)$ is \emph{not} necessarily stable under the
antipodal involution, and in the case $k=0$, the antipodal involution takes $F(c,0,\ell)$ to $F(\ell-c,0,\ell)$.

\begin{thm} \label{thm:dpvariety}
Let $F_n \subseteq \Z^2$ be the set of $(c,\ell)$ where $0 \le \ell \le n+1$
satisfying one of the following two conditions:
\begin{enumerate}
\item $\ell \le \frac{n}{2},\ \ell-\frac{n}{4} \le c \le \frac{n}{4}$, or
\item $\ell \ge \frac{n}{2}+1,\ \frac{n+2}{4} \le c \le \ell - \frac{n+2}{4}$.
\end{enumerate}
Then the collection of line bundles $\O(D)$ for all $D$ in $F(c,0,\ell)$ for all $(c,\ell)$ in $F_n$ coincides with the set $\oldF_n$ and
forms a strong $(S_{n+1} \times C_2)$-stable exceptional collection under any ordering of the
blocks such that $\ell$ is (non-strictly) decreasing.
\end{thm}

The fact that the collection described above coincides with $\oldF_n$ given in Theorem~\ref{thm:Vn} is obvious.  The remainder of this subsection is devoted to proving the above theorem.

\begin{figure}

\begin{subfigure}[t]{0.4\textwidth}
\begin{tikzpicture}[thick,scale=0.7, every node/.style={scale=0.7}]

        \draw[step=1,help lines,black!20] (-1.95,-0.95) grid (3.95,5.95);
        \draw[thick,->] (-1,0) -- (4,0) node [above right] {$c$};
        \draw[thick,->] (0,-1) -- (0,6)node [right] {$\ell$};

        \foreach \Point/\PointLabel in {(0,0)/, (1,2)/, (1,3)/, (2,3)/}
        \draw[fill=black] \Point circle (0.1) node[above right] {$\PointLabel$};

    \end{tikzpicture}
    \centering
     \caption{The set $F_2$}\label{fig:graph2}
    \end{subfigure}
  \begin{subfigure}[t]{0.4\textwidth}
    \begin{tikzpicture}[thick,scale=0.7, every node/.style={scale=0.7}]

        \draw[step=1,help lines,black!20] (-1.95,-0.95) grid (3.95,5.95);
        \draw[thick,->] (-1,0) -- (4,0) node [above right] {$c$};
        \draw[thick,->] (0,-1) -- (0,6)node [right] {$\ell$};

        \foreach \Point/\PointLabel in {(-1,0)/, (0,0)/, (1,0)/, (0,1)/, (1,1)/, (1, 2)/, (2, 4)/, (2, 5)/, (3,5)/ }
        \draw[fill=black] \Point circle (0.1) node[above right] {$\PointLabel$};

    \end{tikzpicture}
    \centering
        \caption{The set $F_4$}\label{fig:graph4}
        \end{subfigure}
    \end{figure}
    
Recall that if $L_1, L_2$ are line bundles,
then $\Ext^i(L_1,L_2) = H^i(X,L_1^{-1} \otimes L_2)$.
Thus, in order to show that collections of line bundles are strongly
exceptional, we need to understand the cohomology of lines bundles.

\begin{lem} \label{lem:acyclic}
Suppose $k \le \ell$.
Then, for all non-zero $D$ in $F(c,k,\ell)$, the line bundle $\O(D)$
satisfies $H^i(X,\O(D))=0$ for all $i$ provided the following conditions hold:
\begin{align*}
-\frac{n}{2}+\ell \le&\ c\ \le \frac{n}{2}-k
&&\textrm{ if } \ell \le \frac{n}{2}\\
1 \le &\ c\ \le \ell -k-1
&&\textrm{ if } \ell  > \frac{n}{2} \ .
\end{align*}
\end{lem}

\begin{proof}
We use the notion of \emph{forbidden cone} introduced by
Borisov and Hua in Section 4 of \cite{BH}
(see also \cite{Efimov}).
Given a subset $I \subseteq \Delta(1)$, we obtain a simplicial complex
$C_I$ where $J \in C_I$ if and only $J \subseteq I$
and there is a cone $\sigma_J$ in $\Delta$ with rays $J$.
If $C_I$ has non-trivial reduced homology, then we obtain
a \emph{forbidden cone} in $\Pic(X)$ given by
\[
F_I=
\left\{ \sum_{\rho \in I} (-1-r_\rho)[D_\rho] + \sum_{\rho \notin I} r_\rho [D_\rho]
\ \middle| \ r_\rho \in \mathbb{R}, r_\rho > 0 \right\}.
\]
A line bundle $L$ satisfies $H^i(X,L)$ for all $i$ if it does not lie in
any forbidden cone (but not conversely, in general).

Note that when $I=\Delta(1)$, we see that $F_I$ is the effective cone,
which explains the criterion for $H^0$.
Unlike us,
Borisov and Hua restrict the definition of ``forbidden cones'' to
\emph{proper} subsets of $I$, but we find including the effective cone
to be more useful in our situation.

Let us use the basis
\[
\left\{\Omega = \left( \sum_{i=0}^{n} E_i -H \right), E_0, \ldots, E_{n} \right\}
\]
for $\Pic(V_n)$, and let $\beta$ be a divisor class in $F(c,k,l)$.
Let $b_\natural$ be the coefficient of $\Omega$ in $\beta$ and $b_i$ the
coefficient of $E_i$ in $\beta$.
Thus $b_\natural=c$ and the other coefficients are in $\{-1,0,1\}$.
Let $L,N,K$ denote the subset of $\{0,\ldots,n\}$ corresponding to
the number of $-1$'s ,$0$'s, and $1$'s, respectively.  Note that $k=|K|$,
$\ell=|L|$ and $|N|=n+1-\ell-k$.

In this basis, $e_i \mapsto E_i$ and $\bar{e}_i \mapsto E_i- \Omega$.
Let us compute the possible coefficients $b_i$ in a forbidden cone with
indexing set $I$.  We use the notation $r_i$ for $r_{e_i}$ and
$\bar{r}_i$ for $r_{\bar{e}_i}$.
For each $i \in \{0,\ldots, n\},$ we have the following table:
\begin{center}
\begin{tabular}{ccc}
$\subseteq I$ & $b_i$ & $b_\natural$ \\
\hline
$\varnothing$ & $r_i+\bar{r}_i$ & $-\bar{r}_i$ \\
$\{e_i\}$ & $-1-r_i+\bar{r}_i$ & $-\bar{r}_i$ \\
$\{\bar{e}_i\}$ & $-1+r_i-\bar{r}_i$ & $1+\bar{r}_i$ \\
$\{e_i,\bar{e}_i\}$ & $-2-r_i-\bar{r}_i$ & $1+\bar{r}_i$
\end{tabular}
\end{center}
where the $b_\natural$ column really just records the contribution from $e_i$ and
$\bar{e}_i$.

\begin{figure}

\begin{subfigure}[t]{0.4\textwidth}

  \begin{tikzpicture}[thick,scale=0.7, every node/.style={scale=0.7}]

        \draw[step=1,help lines,black!20] (-1.95,-0.95) grid (5.95,7.95);
        \draw[thick,->] (-1,0) -- (6,0) node [above right] {$c$};
        \draw[thick,->] (0,-1) -- (0,8) node [right] {$\ell$};

        \foreach \Point/\PointLabel in {(-1,0)/, (0,0)/, (1,0)/, (0,1)/, (1,1)/, (1, 2)/, (2, 4)/, (2, 5)/, (2,6)/, (2, 7)/, (3,5)/, (3, 6)/, (3, 7)/, (4,6)/, (4, 7)/, (5,7)/ }
        \draw[fill=black] \Point circle (0.1) node[above right] {$\PointLabel$};

    \end{tikzpicture}
   \centering
     \caption{The set $F_6$}\label{fig:graph6}
    \end{subfigure}
  \begin{subfigure}[t]{0.4\textwidth}
     \begin{tikzpicture}[thick,scale=0.7, every node/.style={scale=0.7}]

        \draw[step=1,help lines,black!20] (-2.95,-0.95) grid (6.95,9.95);
        \draw[thick,->] (-3,0) -- (7,0) node [above right] {$c$};
        \draw[thick,->] (0,-1) -- (0,10) node [right] {$\ell$};

        \foreach \Point/\PointLabel in {(-2,0)/, (-1, 0)/, (-1, 1)/, (0,0)/, (0,1)/, (0,2)/, (1,0)/, (1,1)/, (1, 2)/, (1, 3)/, (2,0)/, (2,1)/, (2,2)/, (2,3)/, (2,4)/, (3, 6)/, (3, 7)/, (3,8)/, (3,9)/, (4, 7)/, (4, 8)/, (4,9)/, (4, 8)/, (4, 9)/, (5, 8)/, (5, 9)/, (6, 9)/  }
        \draw[fill=black] \Point circle (0.1) node[above right] {$\PointLabel$};

    \end{tikzpicture}
    \centering
        \caption{The set $F_8$}\label{fig:graph8}
        \end{subfigure}
    \end{figure}

Note that if both $e_i$ and $\bar{e}_i$ are contained in
$I$ then the corresponding forbidden cone $F_I$ has $b_i \le -2$. No divisor of $F_n$ lies in a forbidden cone $F_I $ where $I$ contains both $e_i$ and $\bar{e}_i$.
Thus, if our line bundle has non-trivial $H^i$, then it must be contained in a forbidden
cone corresponding to $I$ where $e_i$ and $\bar{e}_i$ do not both appear.
Note that the values of $r_i$ are irrelevant for $b_\natural$.
With this in mind, we determine the possible values of the contributions
to $b_\natural$ given a fixed value of $b_i$. 
They are:
\begin{center}
\begin{tabular}{cccc}
$b_i$ & $\subseteq I$ & $\bar{r}_i$ & $b_\natural$\\
\hline
$1$ & $\varnothing$ & $[0,1]$ & $[-1,0]$\\
 & $\{e_i\}$ & $[2,\infty)$ & $(-\infty,-2]$\\
 & $\{\bar{e}_i\}$ & $[0,\infty)$ & $[1,\infty)$\\
\hline
$0$ & $\varnothing$ & $\{0\}$ & $\{0\}$\\
 & $\{e_i\}$ & $[1,\infty)$ & $(-\infty,-1]$\\
 & $\{\bar{e}_i\}$ & $[0,\infty)$ & $[1,\infty)$\\
\hline
$-1$ & $\varnothing$ & $\varnothing$ & $\varnothing$\\
 & $\{e_i\}$ & $[0,\infty)$ & $(-\infty,0]$\\
 & $\{\bar{e}_i\}$ & $[0,\infty)$ & $[1,\infty)$\\
\end{tabular}
\end{center}

Recall that a \emph{primitive collection} is a minimal subset $C$ of the set of rays
$\Delta(1)$ such that $C$ is not contained in any cone of $\Delta$.
In \cite[Lem. 4.4]{Efimov}, Efimov shows that every (nonempty) indexing set for a
forbidden cone is a union of primitive collections.
Recall that the maximal cones of $\Delta$ contain a set of rays of the form
$\{ e_i \}_{i \in A} \cup \{ \bar{e}_i \}_{i \in B}$
where $A$ and $B$ are disjoint subsets of $\{0, \ldots, n\}$, each
of cardinality $\frac{n}{2}$.
Thus the primitive collections are of the form $\{e_i,\bar{e}_i\}$,
$\{ e_i \}_{i \in S}$ and $\{ \bar{e}_i \}_{i \in S}$
where $S$ is a subset of $\{0, \ldots, n\}$ of cardinality $\ge \frac{n}{2}+1$.

Above we saw that $\beta$ does not lie in any forbidden cone
whose indexing set contains $\{e_i,\bar{e}_i\}$ for any $i$.
Thus we may assume that either $I = \{e_i\}_{i \in S}$
or $I = \{\bar{e}_i\}_{i \in S}$ where $S$ is a subset of
$\{0, \ldots, n\}$ of cardinality $\ge \frac{n}{2}+1$ or $S$ is empty.
Note that we may assume $L \subseteq S$ since there is no way to
produce negative $b_i$ for $i \in L \setminus S$.
We get contributions to $b_\natural$ for each element of various subsets as
follows:
\begin{center}
\begin{tabular}{ccc}
 & $I = \{e_i\}_{i \in S}$ & $I = \{\bar{e}_i\}_{i \in S}$\\
\hline
$L$ & $(-\infty,0]$ & $[1,\infty)$ \\
$K \cap S$ & $(-\infty,-2]$ & $[1,\infty)$ \\
$K \setminus S$ & $[-1,0]$ & $[-1,0]$ \\
$N \cap S$ & $(-\infty,-1]$ & $[1,\infty)$ \\
$N \setminus S$ & $\{0\}$ & $\{0\}$ \\
\end{tabular}
\end{center}

First, let us assume $I = \varnothing$, i.e., the corresponding
forbidden cone is the effective cone.
Here $S = \varnothing$.
To be forbidden, we require $L = \varnothing$.  Our standing assumption is
that $k \le \ell$ so $K = \varnothing$ as well.
It follows that $b_\natural=0$, and the trivial line bundle is the only one of the form $F(c,k,\ell)$
lying in the effective cone.

Now, we assume that $I = \{e_i\}_{i \in S}$.
We see that
\[ b_\natural \le -2|K \cap S| -|N \cap S|. \]
We want to select $S$ to forbid as much as possible.
If $\ell > \frac{n}{2}$ then we may select $S=L$ and we forbid $c \le 0$.
If $\ell \le \frac{n}{2}$ then, since $k \le \ell$, the weakest bound
is obtained by selecting $S$
such that $|N \cap S|=\frac{n}{2}+1-\ell$ where we forbid
$c \le -\frac{n}{2}+\ell-1$.  Indeed, to maximize $-2|K \cap S| - |N \cap S|$, we take $S$ to have minimal size: $|S| = \frac n2 + 1$. Since $L \subseteq S$, we have $|S\setminus L| = \frac n 2 + 1 - \ell$. Note that $|N| = n + 1 - k -\ell$ and $|N| - |S\setminus L| = \frac n 2 + k >0$. Thus, the maximum occurs when $|N \cap S | = |S \setminus L| = \frac{n}{2}+1-\ell$.

Now we assume that $I = \{\bar{e}_i\}_{i \in S}$.
We see that
\[ b_\natural \ge |L| + |K \cap S| - |K \setminus S| + |N \cap S| . \]
Or, since $L \subseteq S$, we have $b_\natural \ge |S| - |K \setminus S|$.
Again, we want to select $S$ so as to forbid as much as possible.
If $\ell \le \frac{n}{2}$ then since $k \le \ell$, we may select
$|S|=\frac{n}{2}+1$ of minimal size, and $K \cap S = \varnothing$.
Thus we forbid $c \geq \frac{n}{2}+1-k$.
If $\ell > \frac{n}{2}$ then we may select $S=L$ so $K \cap S = \varnothing$.
Thus we forbid $c \geq \ell-k$.

We have checked all possible forbidden cones and the statement of the
theorem describes precisely those bundles which are left over.
\end{proof}

In order to build an exceptional collection, we will need to compute
Ext-groups.  Since we are only using line bundles, it suffices to show
that line bundles corresponding to differences of divisors have trivial
cohomology.
Thus, we need the following:

\begin{lem} \label{lem:difference}
If $L_1 \in F(c_1,0,\ell_1)$ and $L_2 \in F(c_2,0,\ell_2)$, then
$L_1-L_2 \in F(c_1-c_2,\ell_2-i,\ell_1-i)$ for an integer $i$ satisfying
\begin{itemize}
\item $i \ge 0$, 
\item $i \le \ell_1$,
\item $i \le \ell_2$, and
\item $i \ge \ell_1+\ell_2-n-1$.
\end{itemize}
\end{lem}

\begin{proof}
The line bundle $L_1$ has the form
\[
c_1\left(\sum_{j=0}^{n}E_j - H\right) - \sum_{j \in J_1}E_j
\]
for some subset $J_1 \subseteq \{0,\ldots,n\}$ of size $\ell_1$.
Similarly, there is a subset $J_2$ of size $\ell_2$.
Their difference is given by
\[
(c_1-c_2)\left(\sum_{j=0}^{n}E_j - H\right)
+ \sum_{j \in J_2 \setminus J_1}E_j
- \sum_{j \in J_1 \setminus J_2}E_j \ .
\]
Note that if $i=|J_1 \cap J_2|$, then
$|J_2 \setminus J_1|=\ell_2-i$ and $|J_1 \setminus J_2|=\ell_1-i$.
The inequalities for $i$ in the statement of the lemma are 
obtained by noting that
$|J_1 \cap J_2|$,
$|J_1 \setminus J_2|$, and
$|J_2 \setminus J_1|$ must be non-negative
and that $|J_1 \cup J_2| \le n+1$.
\end{proof}

\begin{proof}[Proof of Theorem~\ref{thm:dpvariety}]
We show that for any two pairs $(c_1,\ell_1)$, $(c_2,\ell_2) \in F_n$ with $\ell_1 \ge \ell_2$, and for any $L_1 \in F(c_1,0,\ell_1)$
and $L_2 \in F(c_2,0,\ell_2)$, we have
$\Ext^i(\O(L_2),\O(L_1)) \cong H^i(X,\O(L_1-L_2))=0$ for all $i$  
unless $L_1=L_2$.
This will suffice to prove the theorem.
Indeed, taking $\ell_1=\ell_2$ and $c_1=c_2$ shows that
each $S_{n+1}$-orbit is internally orthogonal, taking $\ell_1=\ell_2$ and
$c_2=\ell_1-c_1$ establishes that the whole $(S_{n+1} \times C_2)$-orbit
is orthogonal, and the orbits ordered as in the statement of
the theorem thus form a strong exceptional collection.

Let $L'=L_1-L_2$.
By Lemma~\ref{lem:difference} we know $L' \in F(c_1-c_2,\ell_2-i,\ell_1-i)$
for some $i$.
We will consider 3 distinct cases. 
\begin{itemize}
\item Case 1: $\ell_1, \ell_2 \le \frac{n}{2}$.
For $j = 1, 2$, we have $\ell_j-\frac{n}{4} \le c_j \le \frac{n}{4}$. Adding the inequality for $j = 1$ with the negation of the inequality for $j = 2$, it follows that
\[
\ell_1-\frac{n}{2}
\le c_1 - c_2 \le
\frac{n}{2} - \ell_2 \ .
\]
Thus, for any non-negative $i$, we have
\[
-\frac{n}{2}+(\ell_1-i)
\le c_1 - c_2 \le
\frac{n}{2} - (\ell_2-i).
\]
We conclude that, regardless of $i$, the line bundle $L'$ has trivial
cohomology by Lemma~\ref{lem:acyclic}.

\item Case 2: $\ell_1, \ell_2 > \frac{n}{2}$.  For $j = 1, 2$ we have $\frac{n+2}{4} \le c_j \le \ell_j-\frac{n+2}{4}$. Adding $\frac{n+2}{4}$, we obtain $\frac{n}{2}+1 \le c_j+\frac{n+2}{4} \le \ell_j$.
Thus
\[
-\ell_2+\frac{n}{2}+1
\le c_1 - c_2 \le
\ell_1-\frac{n}{2}-1 \ .
\]
Since $\ell_j> \frac{n}{2}$, we have that $\ell_1+\ell_2 \ge n+1$.
We may assume $\ell_1+\ell_2-n-1 \le i$ from Lemma~\ref{lem:difference}.
Rearranging, we have $\ell_1 -i \le n+1-\ell_2 \leq \frac{n}{2}$.
We also obtain
\[
-\frac{n}{2}+(\ell_1-i)
\le c_1 - c_2 \le
\frac{n}{2}-(\ell_2-i)
\]
using $\ell_1+\ell_2-n-1 \le i$ and conclude that $L'$ has trivial
cohomology once again.

\item Case 3: $\ell_1 > \frac{n}{2}$ but $\ell_2 \le \frac{n}{2}$.
Now $\frac{n+2}{4} \le c_1 \le \ell_1 - \frac{n+2}{4}$
and  $\ell_2-\frac{n}{4} \le c_2 \le \frac{n}{4}$, so
\[
\frac{1}{2}
\le c_1 - c_2 \le
(\ell_1-\ell_2)-\frac{1}{2} \ .
\]
In fact, we have $1 \le c_1 - c_2 \le (\ell_1-\ell_2)-1$,
since $c_1,c_2,\ell_1,\ell_2$ are integers.
If $\ell_1-i > \frac{n}{2}$, then the conditions of Lemma~\ref{lem:acyclic}
are satisfied. 
Otherwise, $\ell_1-i -\frac{n}{2} \le 0$, and so using $\ell_1 \ge
\ell_2$ we have
\[
\ell_1-i -\frac{n}{2} \le c_1-c_2 \le \frac{n}{2} -(\ell_2-i)
\]
to again satisfy the conditions of Lemma~\ref{lem:acyclic}.
\end{itemize}
\end{proof}

\section{Generation via windows} \label{section:window_prelim}

It remains to prove that the collection $\oldF_n$ is full, i.e., that it generates the category $\mathsf{D^b}(V_n)$. To do so, we utilize a particular run of the Minimal Model Program (MMP) for $V_n$. The endpoint of this run is $\mathbb{P}^{n}$, and the birational map $\bbP^n \dashrightarrow V_n $ is described above (blow up the torus invariants points of $\bbP^{n}$ and then inductively flip all of the proper preimages of the torus-invariant linear subspaces of dimension $d < m$, where $n = 2m$). To understand how the derived category is affected under such modifications, it will be advantageous to present the process as a variation of GIT quotients of the spectrum of the Cox ring of the blow up of $\bbP^{n}$ using \cite{BFK, Ballard}. 

We begin by recalling the relevant pieces of the theory of windows and
associated semi-orthogonal decompositions and apply these tools to the
case of toric varieties given by GIT quotients. We provide an
application to the centrally symmetric toric Fano varieties described
above in Section~\ref{sec:app_dp}.

Let us recall some definitions and results of \cite{BFK, Ballard} in the
context of a toric action. We establish some notation and conventions to
be used throughout the remainder of the paper.
Simple flips, blow-ups/downs, and fiber space contractions can be described as moving between chambers in the GKZ fan of a projective toric variety. 

We let $W := \mathbb{A}^r = \operatorname{Spec} k[x_1,\ldots,x_r]$ be a vector space and let $G$ be a subtorus of $\mathbb{G}_m^r \subseteq \operatorname{GL}(W)$. We use $\lambda : \mathbb{G}_m \to G$ to denote a one-parameter subgroup (or cocharacter) of $G$ and $\chi : G \to \mathbb{G}_m$ to denote a character of $G$. The abelian group of characters of $G$ is denoted by $\widehat{G}$.

Recall that the semi-stable locus associated to the $G$-equivariant line bundle $\mathcal O(\chi)$ is 
\begin{displaymath}
 W^{\text{ss}}(\chi) = \{w \in W \mid \exists f \in H^0(W, \mathcal O(n\chi))^G \text{ with } n >0 \text{ and } f(w) \neq 0\}.
\end{displaymath}
Note that the unstable locus (i.e., the complement of $W^{\text{ss}}(\chi))$ is given by the following vanishing locus:
\begin{displaymath}
W^{\text{us}}(\chi) = Z( f \mid f \in H^0(W, \O(n \chi))^G, n >0).
\end{displaymath}
Since $W^{\operatorname{ss}}(\chi) = W^{\operatorname{ss}}(m \chi)$
for $m > 0$, we can naturally extend the definition of semi-stable loci to fractional characters $\widehat{G}_{\mathbb{Q}}$. 
We write 
\begin{displaymath}
 W \modmod {\chi} G : = [W^{\text{ss}}(\chi)/G],
\end{displaymath}
where the right-hand side is the usual quotient stack. If this stack is represented by a scheme, this definition agrees with Mumford's GIT quotient \cite[Prop. 2.1.7]{BFK}.
Let 

\begin{align*}
    C^G(W) = \left\{\chi \in \widehat{G} \mid \exists f \in k[x_1,\ldots,x_r] \text{ with } f \neq 0 \text{ and } \right.
    \\ \left. f(g \cdot  x) = \chi (g) ^n f(x) \text{ for some } n > 0 \right\}.
\end{align*}

The set $C^G(W)\otimes _\Z {\real} \subseteq \widehat{G}_{\real}$ admits a fan structure where the interiors of the cones are exactly the subsets of equal semi-stable locus, see e.g. \cite[Prop. 4.1.3]{BFK}.
Assuming that $W\modmod {\chi} G$ is a simplicial and semi-projective toric variety, $C^G(W)$ is the effective cone of $W\modmod {\chi} G$. In this situation it also coincides with the pseudo-effective cone \cite[Lemma 15.1.8]{CLS}.
We denote this fan by $\Sigma_{\text{GKZ}}$. This is called the \emph{GKZ} (or \emph{GIT} or \emph{secondary}) \emph{fan} associated to the action of $G$ on $W$. 

A \emph{chamber} in $\Sigma_{\text{GKZ}}$ is an interior of a maximal cone. A chamber is a \emph{boundary chamber} if its closure intersects the closure of the complement of $C^G(X)_{\real}$. The \emph{empty chamber} is the complement of $C^G(X)_{\real}$. A \emph{wall} is the relative interior of the intersection of the closures of two adjacent chambers.

The fan $\Sigma_{\text{GKZ}}$ parametrizes the birational models of the usual (scheme-theoretic) GIT quotients that arise via GIT quotients for characters in chambers. Our interest is how the derived category is affected by varying our linearization across walls in $\Sigma_{\text{GKZ}}$. 

Given a one-parameter subgroup $\lambda$, we obtain a linear function
\begin{align*}
 \widehat{\lambda} : \widehat{G} &\to \operatorname{End}(\mathbb{G}_m) \cong \Z \\
 \chi & \mapsto \chi \circ \lambda.
\end{align*}
Each wall in $\Sigma_{\text{GKZ}}$ is the interior of a full-dimensional cone inside the
hyperplane given by the vanishing of some $\widehat{\lambda}$. Denote the corresponding wall by $\Sigma^0_\lambda$. We let $\Sigma^{\pm}_\lambda$ be the
two adjacent chambers lying in the half-spaces $\pm
\widehat{\lambda}_{\real} > 0$, respectively.
Given $x_i$, let $\operatorname{wt}_\lambda(x_i)$ denote $\lambda(a)$
where $x_i$ has grading $a \in \widehat{G}$. Define
\begin{displaymath}
 \mu_\lambda := \widehat{\lambda} \left(\omega_{[W/G]}\right) = -\sum_{i=1}^r \operatorname{wt}_{\lambda} (x_i),
\end{displaymath}
where $\omega _{[W/G]}$ is the canonical sheaf of $[W/G]$.
Without loss of generality,
we will assume that $\mu_{\lambda} \leq 0$. 
Choosing characters $\chi^{\pm} \in \Sigma^{\pm}_\lambda$ and $\chi^0 \in \Sigma^0_\lambda$, set 
\begin{align*}
 X^+_{\lambda} & := W \modmod{\chi^+} G \\
 X^-_{\lambda} & := W \modmod{\chi^-} G \\
 X^0_{\lambda} & := W \modmod{\chi^0} G.
\end{align*}
Denote the following vanishing loci as
\begin{align*}
 W_\lambda^+ & := Z(x_i \mid \operatorname{wt}_{\lambda}(x_i) < 0) \\
 W_\lambda^- & := Z(x_i \mid \operatorname{wt}_{\lambda}(x_i) > 0) \\
 W_\lambda^0 & := Z(x_i \mid \operatorname{wt}_{\lambda}(x_i) \not = 0).
\end{align*}
These are, respectively, the \textit{contracting, repelling}, and
\textit{fixed} loci of the $\gm$-action on $W$ induced by $\lambda$.
Note that $W_\lambda^{\pm}$ is unstable for $\chi$ if $\pm \widehat{\lambda}(\chi) < 0$. 
Finally, set 
\begin{align*}
 Z_{\lambda}^+ & := \left[ \left(W^+_\lambda \cap W^\text{ss}(\chi^+) \right) / G \right] \\
 Z_{\lambda}^- & := \left[ \left(W^-_\lambda \cap W^\text{ss}(\chi^-) \right) / G \right] \\
 Z_{\lambda}^0 & := \left[ \left(W^0_\lambda \cap W^\text{ss}(\chi^0) \right) / G \right].
\end{align*}
These induce the following \emph{wall-crossing diagram}
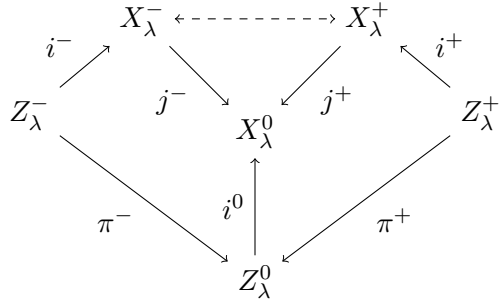
\begin{center}
 \begin{tikzpicture}
  \node at (-1.5,2) (x-) {$X_{\lambda}^-$};
  \node at (1.5,2) (x+) {$X_{\lambda}^+$};
  \node at (-3,0.75) (z-) {$Z_{\lambda}^-$};
  \node at (3,0.75) (z+) {$Z_{\lambda}^+$};
  \node at (0,0.5) (x0) {$X_{\lambda}^0$};
  \node at (0,-1.5) (z0) {$Z_{\lambda}^0$};
  \draw[->] (z-) -- node[above left] {$i^-$} (x-);
  \draw[->] (z+) -- node[above right] {$i^+$} (x+);
  \draw[->] (x-) -- node[below left] {$j^-$} (x0);
  \draw[->] (x+) -- node[below right] {$j^+$} (x0);
  \draw[->] (z-) -- node[below left] {$\pi^-$} (z0);
  \draw[->] (z+) -- node[below right] {$\pi^+$} (z0);
  \draw[->] (z0) -- node[left] {$i^0$} (x0);
  \draw[dashed,<->] (x-) -- (x+);
 \end{tikzpicture}
\end{center}
The maps $j^{\pm}$ are induced by the inclusions $W^\text{ss}(\chi^\pm)
\subseteq W^\text{ss}(\chi^0)$ and $i^{\pm},i^0$ are induced by base
changing the inclusions of $W_{\lambda}^\pm,W^0_{\lambda} \subseteq W$.
The maps $\pi^{\pm}$ are obtained from the projections
$W^{\pm}_{\lambda} \to W^0_{\lambda}$. 

\begin{rem}
The squares in the wall-crossing diagram almost never commute.
Take, for example, 
$X = \mathbb{A}^1$ with its usual action by $\mathbb{G}_m$;
in general, $\lim_{x \to 0} x  = 0 \neq x$.
\end{rem}

Passing to the respective good moduli spaces \cite{Alper}, 
the wall-crossing diagram yields a flip, blow-up/down, or fiber space
contraction diagram of the usual projective toric varieties.
The stacks $Z_{\lambda}^\pm$ are then the exceptional loci.
See Theorem~15.3.13~of~\cite{CLS} in the case of a flip.

The vector space $W_\lambda^0$ carries a trivial $\mathbb{G}_m$-action.
Thus, any quasi-coherent sheaf $\mathcal E$ on $Z_{\lambda}^0$ decomposes as 
\begin{displaymath}
 \mathcal E = \bigoplus_{a \in \Z} \mathcal E_a 
\end{displaymath}
corresponding to the local splitting of the associated $\Z$-graded module into homogeneous summands. 

\begin{defn} 
 We let $\mathcal E_a$ be the {$a$-th $\lambda$-weight space} of $\mathcal E$. We set 
 \begin{displaymath}
  \operatorname{wt}_{\lambda}(\mathcal E) := \{ a \in \mathbb Z \mid \mathcal E_a \not = 0 \}.
 \end{displaymath}
Note that $\operatorname{wt}_{\lambda}(\mathcal O(\chi)) =
\{ \widehat{\lambda}(\chi) \} $; or by a slight abuse of notation,
$\operatorname{wt}_{\lambda}(\mathcal O(\chi)) = \widehat{\lambda}(\chi)$.
For any $I \subseteq \mathbb Z$, let $\mathsf C_{\lambda}(I)$ denote the full subcategory of $\mathsf{D^b}(Z_{\lambda}^0)$ consisting of objects $\mathcal E$ whose cohomology sheaves have weights contained in $I$:
\begin{displaymath}
 \operatorname{wt}_{\lambda}(\mathcal H^\ast\mathcal E) \subseteq I.
\end{displaymath}
We set $\mathsf C_{\lambda}(a) := \mathsf C_{\lambda}(\{a\})$.
 
 The $I$-\emph{window} associated to $\lambda$, denoted $\weezer_{\lambda,I},$ is the full subcategory of $\mathsf{D^b}(X_\lambda^0)$ consisting of objects $\mathcal E$ whose derived restriction $\left( i^0 \right)^\ast \mathcal E$ lies in $\mathsf C_{\lambda}(I)$. 
\end{defn}

\begin{lem} \label{lem:wall contributions}
Suppose that $\widehat{\lambda}$ is primitive: if $nv=\widehat{\lambda}$
for $v \in \widehat{G}$ and $n \in \Z_{>0}$ then $v=\widehat{\lambda}$.
 For any $a \in \mathbb{Z}$, there is an equivalence
 \begin{displaymath}
  \mathsf{D^b}(Z_{\lambda}^{0,\operatorname{rig}}) \cong \mathsf C_\lambda(a).
 \end{displaymath}
 Moreover, in this case, the rigidification of $Z_\lambda^0$ with respect to $\mathbb{G}_m$ is given by
 \begin{displaymath}
  Z_{\lambda}^{0,\operatorname{rig}} \cong \left[ \left(W^0_\lambda \cap W^{\operatorname{ss}}(\chi^0) \right) / (G/\lambda(\mathbb{G}_m)) \right].
 \end{displaymath}
\end{lem}

\begin{proof}
The second statement is \cite[Section 5.1.3]{cherry}.
Note that we can split
$G \cong G/\lambda(\mathbb{G}_m) \times \lambda(\mathbb{G}_m)$
since $\widehat{\lambda}$ is primitive.
Given the second statement, if we tensor an object $\mathcal E$ of $\mathsf{D^b}(Z_{\lambda}^{0,\text{rig}})$ with $\mathcal O_{\operatorname{Spec} k}(a)$ we get an object of $\mathsf C_\lambda(a)$. This is the inverse to tensoring with $\mathcal O(-a)$ and pushing down via $Z_\lambda^0 \to Z_{\lambda}^{0,\text{rig}}$. 
\end{proof}

We set
\begin{displaymath}
 t^{\pm}_{\lambda} :=
\widehat{\lambda}\left(\omega_{[W^{\mp}_{\lambda}/G] \mid [W/G]}\right)
= -\sum_{ \pm \operatorname{wt}_{\lambda} (x_i) > 0 } \operatorname{wt}_{\lambda} (x_i) 
\end{displaymath}
and 
\begin{align*}
 I_{d,\lambda}^{\pm} & := [d \pm t_{\lambda}^{\pm}, d-1]
\end{align*}
where $d$ is an arbitrary integer.
Note that $\mu_{\lambda} = t^+_{\lambda} + t^-_{\lambda}$.
Windows allow one to ``lift" derived categories, made precise by the following fundamental results.

\begin{prop}[Fundamental Theorem of Windows I, Cor. 2.23 \cite{Ballard}, see also \cite{HalpLeist}]\label{prop:FTW1}
 The functors 
 \begin{displaymath}
  \left( j^{\pm} \right)^\ast |_{\weezer_{\lambda,I^{\pm}_d}} : \weezer_{\lambda,I^{\pm}_{d,\lambda}}  \to \mathsf{D^b}(X_\lambda^\pm) 
 \end{displaymath}
 are equivalences. 
\end{prop}

\begin{defn}
Since $\weezer_{\lambda,I^\pm_{d,\lambda}}$ is a full subcategory of
$\mathsf{D^b}(X_{\lambda}^0)$, we may define fully faithful functors
 \begin{displaymath}
  Q_d^\pm: \mathsf{D^b}(X_{\lambda}^\pm) \to \mathsf{D^b}(X_{\lambda}^0) 
 \end{displaymath}
as the inverse to $\left(j^\pm\right)^\ast |_{\weezer_{\lambda,I^\pm_{d,\lambda}}}$
followed by inclusion. That is, we have a diagram

\begin{center}
\begin{tikzpicture}
  \node at (0,0) (w) {$\weezer_{\lambda, I^+} $};
  \node at (-1.5,1.5) (x+) {$\mathsf{D^b}(X^+_{\lambda})$};
  \node at (1.5,1.5) (x0) {$ \mathsf{D^b}(X_{\lambda}^0)$};

  \draw[->] (x+) -- node[above] {$Q^+_d$} (x0);
  \draw[->] (w) -- node[below left] {$(j^+)*$} (x+);
  \draw[->] (w) -- node[below right] {$i$} (x0);
 \end{tikzpicture}
 \end{center}
 We also define
 \begin{align*}
  \Phi_d & := \left( j^- \right)^\ast \circ Q^+_d \\
  \Psi_d & := \left( j^+ \right)^\ast \circ Q^-_d.
 \end{align*}
\end{defn}

\begin{rem}
 Note that $\left( j^\pm \right)^\ast \circ Q^\pm_d \cong \mathds{1}_{\mathsf{D^b}(X_{\lambda}^\pm)}$, so that $Q^{\pm}_d$ is a right inverse to $\left( j^\pm \right)^\ast$.
\end{rem}

The following describes the effect that passing through a wall has on the corresponding derived categories.

\begin{thm}[Fundamental Theorem of Windows II, Thm. 2.29 \cite{Ballard}] \label{thm:fundwin2}
 For any $d \in \mathbb Z$, there is a semi-orthogonal decomposition
 \begin{displaymath}
  \mathsf{D^b}(X_{\lambda}^+) = \left\langle \mathsf C_{\lambda}(d + t^+_\lambda), \mathsf C_{\lambda}(d + 1 + t^+_\lambda), \ldots, \mathsf C_{\lambda}(d - 1 - t^-_\lambda), \mathsf{D^b}(X_\lambda^-) \right\rangle
 \end{displaymath}
 where the explicit fully-faithful functors are given by
 \begin{displaymath}
  \Upsilon_a^+ := i^+_\ast \circ \left(\pi^+\right)^\ast : \mathsf C_{\lambda}( a) \to \mathsf{D^b}(X_\lambda^+)
 \end{displaymath}
 and 
 \begin{displaymath}
  \Phi_d : \mathsf{D^b}(X_\lambda^-) \to \mathsf{D^b}(X_\lambda^+). 
 \end{displaymath}
 Moreover,
 \begin{displaymath}
  \Psi_d : \mathsf{D^b}(X_\lambda^+) \to \mathsf{D^b}(X_\lambda^-) 
 \end{displaymath}
 is the right adjoint (projection) functor to $\Phi_d$. 
\end{thm}

\begin{ex} \label{ex:P^1}
It is instructive to consider the case where
$W =\Spec( k[x_1, x_2])$ and $G=\gm$ with $\widehat{G}=\Z$
such that $x_1,x_2$ have degree $1$.
Here we take $\widehat{\lambda} = \id$, $\chi^+=1$, $\chi^-=-1$
and $\chi^0=0$.
We find that $\omega_{[W/G]}=-2$, $\mu_\lambda=-2$,
$t_+=-2$ and $t_-=0$.
Furthermore, $X^+_\lambda = \bbP^1$,
$X^-_\lambda = \varnothing$, and
$X^0_\lambda = [W/G]$.
Lastly, $Z^+_\lambda = \bbP^1$, $Z^-_\lambda \cong \varnothing$
and $Z^0_\lambda \cong BG$.
The First Fundamental Theorem of Windows yields a semi-orthogonal decomposition
$\mathsf{D^b}(X_\lambda^+) =
\langle \O_{\bbP^1}(d-2), \O_{\bbP^1}(d-1) \rangle,$ and
$\mathsf{D^b}(X_\lambda^-) = 0$
for any $d$. This recovers the foundational result of Beilinson. 
Note that this is also compatible with the Second Fundamental Theorem of
Windows.
\end{ex}

The following lemma allows us to track the action of $\Psi_d$ for particular objects. 

\begin{lem} \label{lem:through_the_wall}
 If $E = \left( j^+ \right)^* F$ and $\operatorname{wt}_{\lambda}(F) \subseteq I^+_{d,\lambda}$, then $\Psi_d (E) = \left( j^- \right)^\ast F$. In particular, if $\widehat{\lambda}(\chi) \in I_{d,\lambda}^+$, then 
\begin{displaymath}
 \Psi_d \left( \mathcal O_{X_\lambda^+}(\chi) \right) = \mathcal O_{X_\lambda^-}(\chi).
\end{displaymath}
\end{lem}

\begin{proof}
 Note that $Q_d^+$ sends any object $A$ of $\mathsf{D^b}(X_\lambda^+)$
to the unique object $B$ in $\mathsf{D^b}(X_\lambda^0)$ such that 
$\left(j^+\right)^\ast B = A$ and $\left(i^0 \right)^\ast B$ lies in
$\mathcal C_{\lambda, I^+_{d,\lambda}}$.
 Clearly, by assumption, $F$ satisfies both of these conditions for $E$. Hence 
 \begin{displaymath}
  \Psi_d(E) = \left( j^- \right) ^\ast Q_d^+ E = \left( j^- \right) ^\ast F.
 \end{displaymath}
\end{proof}

We provide a technical lemma before applying the above framework to generalized del Pezzo varieties.
Recall that an object $F$ \emph{generates} a triangulated category $\mathsf{T}$
if the smallest thick subcategory of $\mathsf{T}$ containing $F$ is
$\mathsf{T}$ itself.

\begin{lem}\label{lem:gen}
Let $\mathsf{T}$ be a triangulated category with a given semiorthogonal
decomposition $\mathsf{T} = \langle \mathsf{A}, \mathsf{B} \rangle,$ and
let $\Psi: \mathsf{T} \to \mathsf{B}$ be the right adjoint to the inclusion
$\mathsf{B}  \subseteq \mathsf{T}$.
Assume there is an object $F \in \mathsf{T}$ such that
\begin{enumerate}
\item $\mathsf{A}$ is contained in the subcategory generated by $F$, and
\item $\Psi(F)$ generates $\mathsf{B}$.
\end{enumerate}
Then $F$ generates $\mathsf{T}$.
\end{lem}

\begin{proof}
For any object $C \in \mathsf{T}$, we have a distinguished triangle $C_b \to C \to C_a$, functorial in $C$, where $ C_a \in \mathsf{A}$ and $C_b \in \mathsf{B}$.  Given generators $A$ of $\mathsf{A}$ and $B$ of $\mathsf{B}$, the object $\iota_{\mathsf{A}} (A) \oplus \iota _{\mathsf{B}}(B)$ generates $\mathsf{T}$, where $\iota_{\mathsf{A}}$ and $\iota _{\mathsf{B}}$ are the inclusions of $\mathsf{A}$ and $\mathsf{B}$ in $\mathsf{T}$.  The object $F$ has associated triangle $F_b \to F \to F_a$.  Since $F$ generates $\mathsf{A}$, it generates $F_a$. We can thus generate $F_b$ from $F$ using the distinguished triangle above.  Furthermore $F_a$ generates $\mathsf{A}$.  Note that $F_b = \iota _{\mathsf{B}}(\Psi(F))$, so since $\Psi (F)$ generates $\mathsf{B}$, it follows that $F_b$ generates $\mathsf{B}$.  Since $F_a$ generates $\mathsf{A}$ and $F_b$ generates $\mathsf{B}$, it follows that $F$ generates $\mathsf{T}$. 
\end{proof}

\section{Application of windows to del Pezzo varieties} \label{sec:app_dp}

Recall that $Y_n$ denotes the blow-up of $\bbP^n$ at the $(n+1)$ torus-invariant points. It is a toric variety with torus $T = \gm^n$. The spectrum of the Cox ring of $Y_n$ is isomorphic to $\bbA^{2n+2}$. Choosing a basis for $\operatorname{Div}_T(Y_n)$ consisting of 
\begin{displaymath}
 \left\{\bar{e}_0 := \left(H - \sum _{i \neq 0}E_i\right), \ldots ,\bar{e}_{n} := \left(H- \sum _{i \neq n }E_i \right), e_0 := E_0, \ldots , e_{n}:= E_{n}\right\}
\end{displaymath}
where $E_i$ is the exceptional divisor for the $i$-th point and $H$ is the hyperplane class, the action of the Picard torus $G \cong \gm^{n+2}$ on $W := \bbA^{2n+2}$ gives a $\operatorname{Pic}(Y_n) \cong \Z^{n+2}$-grading to the polynomial ring $k[x_0,\ldots, x_{n}, y_0, \ldots, y_{n}]$
where we have correspondences
$x_i \leftrightarrow \bar{e}_i$ and $y_i \leftrightarrow e_i$. Since we have a $\operatorname{Pic}(Y_n)$-grading, we can and will identify characters of $G$ with elements of $\operatorname{Pic}(Y_n)$. The weight matrix in the basis $H,E_0,\ldots,E_n$ is given by the $(n+2) \times (2n + 2)$-matrix
\[
\gamma := \left(\begin{array}{rrrrrrrrr}
1 & 1 & \cdots & 1 & 0 & 0 & \cdots & 0 & 0\\

0 & -1 & \cdots & -1 & 1 & 0 & \cdots &  0 & 0 \\

-1 & 0 & \cdots & -1& 0 & 1 & \cdots & 0 & 0 \\

\vdots & \vdots & \ddots & \vdots & \vdots & \vdots & \ddots & \vdots & \vdots  \\

-1 & -1 & \cdots & -1 & 0 & 0 & \cdots & 1 & 0\\

-1 & -1 & \cdots & 0 & 0 & 0 & \cdots & 0 & 1 
\end{array}\right) .
\] 
For brevity, we set $E := \sum_{i=0}^n E_i$.

\begin{lem}
 There is an isomorphism
 \begin{displaymath}
  V_n \cong W \modmod{-K} G
 \end{displaymath}
 where $-K = (n+1)H - (n-1) E$ is the anticanonical divisor of $V_n$.
There is also an isomorphism 
 \begin{displaymath}
  Y_n \cong W \modmod{H-t E} G
 \end{displaymath}
 where $0 < t \ll 1$. 
\end{lem}

\begin{proof} 
We first treat the presentation of $V_n$ as a GIT quotient by comparing the description of the associated polytope given in \cite[p. 234]{VosKly} with that given in \cite[\S14.2]{CLS}. Note that the polytope for $-K$ is centrally-symmetric. In \cite{VosKly}, the authors use the polytope $P = \text{Conv}( \pm f_i \mid 0 \leq i \leq n)$ where
$$\begin{array}{rl}
f_0 &= (-1,-1,\cdots,-1)\\
f_1 &= (1,0,\cdots,0)\\
f_2 &= (0,1,\cdots,0)\\
&\vdots\\
f_n &= (0,0,\cdots,1)
\end{array}$$
in $N = \Z^n$. Let $P^{\vee} = (m \mid \pm m _i \geq -1 \text{ for all } 1 \leq i \leq n \text{ and } \pm \sum m_i \geq -1)$ be its dual polytope in $M = N^{\vee}$. 
Turning to the GIT presentation, we have the usual short exact sequence for a
smooth projective toric variety $X$ with fan $\Delta$:
\begin{displaymath}
 0 \to M \to \Z^{\Delta (1)} \to \text{Pic}(X_{\Delta}) \to 0.
\end{displaymath}
Let $T$ be the torus of the toric variety $X_{\Delta}$.
Note that $M = \widehat{T}$ and $\text{Pic}(X_{\Delta}) = \widehat{G}$, the character groups of $T$ and $G$. This sequence may be identified with 
\begin{displaymath}
 0 \to \Z ^{n} \cong \text{ker}(\gamma) \xrightarrow{\delta} \Z^{2n+2} \xrightarrow{\gamma} \Z^{n+2} \to 0
\end{displaymath}
where $\delta$ is the inclusion. The matrix defining $\gamma$ is the weight matrix given above. 
It is clear that the kernel of $\gamma$ is given by 
\[
\ker (\gamma) = \left\{(\alpha, \beta)\ \middle|\ \sum \alpha_i = 0 \text{ and } \beta _j - \sum _{i \neq j} \alpha _i = 0\right\},
\]
where $\alpha = (\alpha_0, ..., \alpha_{n})$ and $\beta = (\beta_0, ..., \beta_{n})$.
One choice of matrix exhibiting the composite
$\Z^n \cong \text{ker}(\gamma) \xrightarrow{\delta} \Z^{2n+2}$ is
the $(2n+2)\times n$-matrix
\[
\left(\begin{array}{rrrr}
1 & 0 & \cdots & 0\\
0 & 1& \cdots & 0\\
\vdots & \vdots & \cdots & \vdots \\
0 & 0 & \cdots & 1\\
-1 & -1 & \cdots & -1\\
-1 & 0 & \cdots & 0\\
0& -1& \cdots & 0\\
\vdots & \vdots & \cdots & \vdots \\
0 & 0 & \cdots & -1\\
1 & 1 & \cdots & 1 \\
\end{array}\right) =
\left(\begin{array}{rrrrr}
& & & & \\
& & \text{Id} & & \\
& & & & \\
-1 & -1& -1 &  \cdots & -1\\
& & & & \\
& & -\text{Id} & &  \\
& & & & \\
1 & 1 & 1 & \cdots & 1\\
\end{array}\right).
\]
We have $\delta (m) = (\langle m, \nu_1 \rangle, ..., \langle m, \nu_{2n+2}\rangle)$ for some $\nu_{1}, ..., \nu_{2n+2} \in M^{\vee} = N$.
Reading the rows, these are precisely the elements of $\Delta (1)$.

As shown in \cite[\S14.2]{CLS}, the polytope of $\mathbb{A}^{2n+2} \modmod{-K}  G$ is given by 
\[
\text{Conv}\left((\alpha_1, \ldots, \alpha_n)\ \middle|\
\alpha_i \geq -1,\ - \sum \alpha_i \geq -1,\
,\ -\alpha_i \geq -1,\ \sum \alpha_i \geq -1\right).
\]
This is precisely the polytope $P^{\vee}$ described above, and the claimed GIT description of $V_n$ is verified.

For the presentation of $Y_n$ as a GIT quotient, we note that $H-tE$ is ample as a line bundle on $Y_n$. Thus, taking the quotient of $W$ relative to $H-tE$ produces a variety on which $H-tE$ (viewed as a bundle on this quotient) is ample, so these descriptions yield isomorphic varieties.
\end{proof}

Knowing that $V_n$ and $Y_n$ occur as GIT quotients via linearizations in different chambers of the secondary fan, we analyze the wall-crossing behavior. Let us begin by identifying the walls. Consider the following cocharacters: for any subset $J \subseteq \{0,1,\ldots,n\}$, define 
\begin{align*}
 \widehat{\lambda}_J : \operatorname{Pic}(Y_n) & \to \Z \\
 H & \mapsto 1-|J| \\
 E_j & \mapsto \begin{cases} -1 & \text{if } j \in J \\ 0 & \text{if } j \notin J \end{cases}
\end{align*}
We denote the corresponding cocharacters by $\lambda_J : \gm \to G$.
Evaluating $\widehat{\lambda}_J$ on $F_{c, L} = c (E - H) - \sum _{i \in L} E_i$ yields
$$\widehat{\lambda} _J(F_{c,L}) = c(-|J| - (1-|J|))+ |L \cap J| = |L \cap J|-c.$$

\begin{lem} \label{lem:GKZwalls} 
The walls in the GIT fan are precisely those given by $\widehat{\lambda}_J = 0$
for $J \subseteq \{0,\ldots,n\}$.
\end{lem}

\begin{proof}
Walls in the GIT fan correspond to circuits in the set $\Delta(1)$ of one-cones \cite[\S15.3, p. 751]{CLS}. Recall that a circuit in $\Delta(1)$ is given by a linearly dependent set of ray generators such that each proper subset is linearly independent \cite[p. 751]{CLS}. There are two ways to obtain such a collection.  The first is to choose one of $e_i$ or $\overline{e}_i$ for each $0 \leq i \leq n$ since $\overline{e}_i = -e_i$ in $N$. In other words, for any subset $J \subseteq \{0,\ldots,n\}$, we take $\overline{e}_i$ for $i \not \in J$ and $e_i$ for $i \in J$. Note that we have the following primitive relation amongst the elements of this circuit:
\begin{displaymath}
 \sum_{i \in J^c} \overline{e}_i - \sum_{i \in J} e_i = 0. 
\end{displaymath}
This is exactly the image of $\widehat{\lambda}_J$.  The other way to obtain a circuit is to take $\{e_i, \overline{e}_i\}$ for any $i \in \{0, ..., n\}$. This is the image of $\widehat{\lambda}_{[\infty, i]}$, which is defined via $H \mapsto 1$ and $E_j \mapsto \delta_{ij}$. 
\end{proof} 

\begin{cor} \label{cor:pn}
 The GIT quotient for the chamber with $\operatorname{wt}_{\lambda_J}(\chi) < 0$ for $|J| > 0$ and $\operatorname{wt}_{\lambda_{\varnothing}}(\chi) > 0$ is isomorphic to $\mathbb{P}^n$.
\end{cor}

\begin{proof}
 We find it useful to the use the involution of $G = \gm^{n+2}$ given by 
 \begin{displaymath}
  (\beta, \alpha_0, \ldots, \alpha_n) \mapsto (\beta^{-1}, \beta \alpha_0, \ldots, \beta \alpha_n) =: (\gamma, \delta_0,\ldots,\delta_n).
 \end{displaymath}
 After applying this involution, the weight matrix becomes 
 \begin{displaymath}
  \left(\begin{array}{rrrrrrrrrr}
-1 & -1 & \cdots & -1 & -1 & 0 & 0 & \cdots & 0 & 0\\

1 & 0 & \cdots & 0 & 0 & 1 & 0 & \cdots &  0 & 0 \\

0 & 1 & \cdots & 0 & 0 & 0 & 1 & \cdots & 0 & 0 \\

\vdots & \vdots & \ddots & \vdots & \vdots & \vdots & \ddots & \vdots & \vdots  \\

0 & 0 & \cdots & 1 & 0 & 0 & 0 & \cdots & 1 & 0\\

0 & 0 & \cdots & 0 & 1 & 0 & 0 & \cdots & 0 & 1 
\end{array}\right)
 \end{displaymath}
 Next we check that if $y_0 = 0$, then the point is unstable for $\chi$. It suffices to exhibit a one-parameter subgroup $\lambda$ of $G$ such that 
 \begin{displaymath}
  \lim_{\alpha \to 0} \alpha \cdot (x_0,x_1,\ldots,x_n,0, y_1,\ldots,y_n)
 \end{displaymath}
 exists and $\operatorname{wt}_{\lambda}(\chi) > 0$. We can use
$\lambda_{\{0\}}^{-1}$ which, by assumption, satisfies
$\operatorname{wt}_{\lambda^{-1}_{\{0\}}}(\chi) > 0$  on the chamber.
Appealing to $S_{n+1}$-symmetry, we see that to be semi-stable it is
necessary that $y_i \not = 0$ for all $i$. Consider the subgroup $H =
\{(1,\delta_0,\ldots,\delta_n)\} \subseteq G$. The $H$-invariant subring
of $k[x_0,\ldots,x_n,y_0^{\pm 1},\ldots,y_n^{\pm 1}]$ is $k[x_0y_0^{-1},
\ldots, x_ny_n^{-1}]$ and it carries a $\gm$-action from $\gamma$ satisfying $\operatorname{wt}_\lambda(x_iy_i^{-1})=-1$. Thus, 
 \begin{displaymath}
  \left[\left(\mathbb{A}^{n+1} \times \gm^{n+1}\right)/G\right] \cong \left[\mathbb{A}^{n+1}/\gm\right].
 \end{displaymath}
 Passing to the semi-stable locus gives $\mathbb{P}^n$. 
\end{proof}

The geometric manifestation of Lemma~\ref{lem:GKZwalls} is exactly the
description relating $\mathbb{P}^n$ and $V_n$ via flips and then a
blow-down to $\mathbb{P}^n$. When crossing a wall (in the $K$-positive
direction) corresponding to $\frac{(n+1)}{2} \geq |J|\geq 2$, we flip
the contracting locus of $\lambda_J$ for the repelling locus of
$\lambda_J$. When crossing a wall for $J$ with $|J|=1$, we blow down the
corresponding exceptional divisor. Finally, when $J = \varnothing$, we
contract $\mathbb{P}^n$ down to a point. Each of these
exceptional/flipping loci are quotients of linear subspaces of $W = \mathbb{A}^{2n+2}$.

\begin{figure}[ht]
     \begin{tikzpicture}[thick,scale=0.7, every node/.style={scale=0.8}]

        \draw[step=1,black!20] (-4.95, .05) grid (10.95,11.95);
        \draw[thick,->] (3,1) -- (9,1) node [above right] {$\widehat{\lambda}_{\varnothing}$};
        \draw[thick,->] (3,1) -- (3,11) node [right] {$\widehat{\lambda}_{[1]}$};
        \draw[thick,->] (3,1) -- (-3,8.5) node [above left] {$\widehat{\lambda}_{[m]}$};
        \draw[thick,->] (3,1) -- (-1, 10) node [above left] {$\widehat{\lambda}_{[m-1]}$};
        \draw[thick,->] (3,1) -- (0.5,10.5) node [above] {$\widehat{\lambda}_{[2]}$};
  	
	 \draw[thick,->] (3,1) -- (-4,7) node [above] {$\widehat{\lambda}_{[n]}$};
	 
	  \draw[thick,->] (3,1) -- (-4,5) node [above] {$\widehat{\lambda}_{[\infty]}$};
	
	\draw[] (0,9) node [above] {$\iddots$};
	
	\draw[] (-3,7) node [above] {$\iddots$};
	
	\draw[] (2,8) node [above] {$H-tE$};
	
	\draw[] (2,10) node [above] {$Y_n$};
	
	\draw[] (0,7) node [below left] {$-K$};
	
	\draw[] (-2,9) node [below right] {$V_n$};
	
	\draw [->,
line join=round,
decorate, decoration={
    zigzag,
    segment length=8,
    amplitude=.9,post=lineto,
    post length=2pt
}]  (-0.5,7) -- (-1.25,8.25);

	\draw [->,
line join=round,
decorate, decoration={
    zigzag,
    segment length=8,
    amplitude=.9,post=lineto,
    post length=2pt
}]  (2,8.5) -- (2,10);

\draw[] (5,3) node [below] {$H+ tE$};

\draw[] (7,5) node [above right] {$\mathbb{P}^n$};

	\draw [->,
line join=round,
decorate, decoration={
    zigzag,
    segment length=8,
    amplitude=.9,post=lineto,
    post length=2pt
}]  (5,3) -- (7,5);

\draw[] (-2,3) node [] {$\varnothing$};

\draw[] (6,0) node [] {$\varnothing$};

        \foreach \Point/\PointLabel in {(3,1)/  }
        \draw[fill=black] \Point circle (0.1) node[above right] {$\PointLabel$};

    \end{tikzpicture}
  \caption{Portion of $\Sigma_{\text{GKZ}}$ generated by $E$ and $H$, and walls relating $\mathbb{P}^n$, $Y_n$, and $V_n$ and their linearizations.}
\label{fig:flippingwalls}
\end{figure}
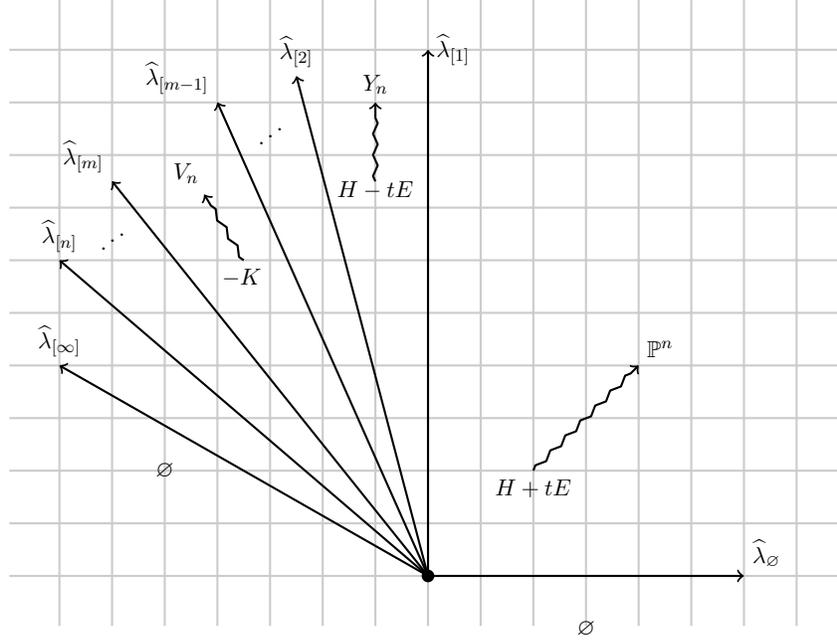

\begin{rem}
To get a sense of $\Sigma_{\text{GKZ}}$, consider the plane spanned by
$E$ and $H$. This is presented in Figure~\ref{fig:flippingwalls}. Note that the walls determined by $\widehat{\lambda}_J$ for equal $|J|$
coincide.
Here $\widehat{\lambda}_{[a]}$ denotes the intersection of $\widehat{\lambda}_J$ with
$\operatorname{span}\{E,H\}$ when $a=|J|$, and $\widehat{\lambda}_{[\infty]}$ denotes the intersection of $\widehat{\lambda}_{[\infty, i]}$ (as defined in the proof of Lemma~\ref{lem:GKZwalls}) with $\operatorname{span}\{E,H\}$.

Notice that the anti-canonical divisor $-K$ lies on the ray emanating from the origin and passing through $(-(n-1), n+1)$.  Taking absolute values of the slopes of $\widehat{\lambda}_{[m]}$, $\widehat{\lambda}_{[m-1]}$, and the line passing through $-K$, the inequalities
\[
\frac{m+1}{m} \leq \frac{n+1}{n-1} \textrm{ and }
\frac{n+1}{n-1} \leq \frac{m}{m-1}
\]
show that $-K$ lies above $\widehat{\lambda}_{[m]}$
and below $\widehat{\lambda}_{[m-1]}$.
\end{rem}

We need to identify the contracting and repelling loci associated to each $\lambda_J$. The following result verifies the above claim that these loci are linear subspaces of $W$.

\begin{lem}
 On $W$, the ideal of the contracting locus
$W_J^+ := W_{\lambda_J}^+$ is $\left(y_j \mid j \in J\right)$, and the ideal of the
repelling locus $W_J^- := W^-_{\lambda_J}$ is $\left(x_i \mid i \not \in J \right)$.
The ideal of the fixed locus $W_J^0 := W^0_{\lambda_J}$ is $\left(x_i,y_j \mid i \not \in J, j \in J\right)$.
\end{lem}

\begin{proof}
 This is obvious from the definitions. 
\end{proof}

In light of Theorem~\ref{thm:fundwin2}, we also record $t^{\pm}$ and $\mu$ for each $J$. 

\begin{lem} \label{lem:wall_weight}
Let $t_J^{\pm} := t^{\pm}_{\lambda _J}$ and $\mu_J := \mu(\lambda _J)$ for $ J \subseteq \{0, ..., n\}$. We have the following: 
\begin{itemize}
\item $t^{+}_{J} =
\widehat{\lambda}\left(\omega_{[W^{-}_{J}/G] \mid [W/G]}\right)
= -|J^c|$, 
\item $t^{-}_{J} =
\widehat{\lambda}\left(\omega_{[W^{+}_{J}/G] \mid [W/G]}\right)
= |J| $. 
\end{itemize}
Hence,
$ \mu_J := t_J^+ + t_J^- = |J| - |J^c| = 2|J| - n - 1.$
 
\end{lem}

\begin{proof}
 This follows from
 \begin{displaymath}
  \widehat{\lambda}_J(e_i) = \begin{cases} -1 & \text{if } j \in J \\ 0 & \text{if } j \notin J \end{cases}
 \end{displaymath}
 and 
 \begin{displaymath}
  \widehat{\lambda}_J(\bar{e}_i) = \begin{cases} 0 & \text{if } j \in J \\ 1 & \text{if } j \notin J .\end{cases}
 \end{displaymath}
\end{proof}

The following lemma is the key observation in proving that $\oldF_n$ generates $\mathsf{D^b}(V_n)$. Given a set $J \subseteq \{0,\ldots,n\}$, we denote the Koszul complex associated to the set $\{y_i \mid i \in J\}$ by $K(J)$. 

\begin{lem} \label{lemma:gen_koszul}
 Let $J \subseteq \{0,\ldots,n\}$ with $|J| \leq \frac n2$. Let $w$ be an integral point in the interval 
  \begin{displaymath}
   \left[ \frac{3n}{4} + |J| - n, \frac{3n+2}{4} - |J|\right] = \left[ \frac{3n+4}{4} - |J^c|, |J^c| - \frac{n+2}{4} \right].
  \end{displaymath}
  \begin{itemize}
    \item If $w \leq \frac n4$, the components of the tensor product $K(J) \otimes \mathcal O(w(E-H))$ lie in $\oldF_n$.
    \item If $w \geq \frac{(n+2)}{4}$, the components of the tensor product $K(J) \otimes \mathcal O(w(E-H) - \sum_{i \in J^c} E_i)$ lie in $\oldF_n$.
 \end{itemize}
\end{lem}

\begin{proof}
 Using the action of $S_{n+1}$, we may assume that $J =
\{0,1,\ldots,|J|-1\}$. Each component of the Koszul complex $K(J)$ is
given by $\mathcal O(- \sum_{i \in L} E_i)$ for some $L \subseteq J$.
Assume $w \leq \frac n4$ and note that the tensor product $\mathcal O(-
\sum_{i \in L} E_i) \otimes \mathcal O(w(E-H))$ is precisely the line
bundle $\O(F_{w,L})$. Thus, it suffices to check that $(w,\ell) \in F_n$ whenever $0 \leq  \ell \leq |J|$ and 
 \begin{displaymath}
  \frac{3n}{4} + |J| - n \leq w \leq \frac{3n+2}{4} - |J|.
 \end{displaymath}
Simplifying the first inequality yields $|J| -\frac n4 \leq w$.  Using the assumption that $w \leq \frac n4$ together with the fact that $|L| \leq |J|$, we have
 \begin{displaymath}
  |L| - \frac{n}{4} \leq |J| - \frac{n}{4} \leq w \leq \frac n4.
 \end{displaymath}
 Thus, $\O(F_{w,L}) \in \oldF_n$ by definition. Assume that $w \geq \frac{n+2}{4}$ and note the tensor product $\mathcal O(- \sum_{i \in L} E_i) \otimes \mathcal O(w(E-H) - \sum_{i \in J^c} E_i)$ is the line bundle $F_{w,L \cup J^c}$. Thus, we need to check that $(w,\ell+|J^c|) \in F_n$ whenever $0 \leq \ell \leq |J|$ and 
 \begin{displaymath}
  \frac{3n}{4} + |J| - n \leq w \leq \frac{3n+2}{4} - |J|.
 \end{displaymath}
 Note that 
 \begin{displaymath}
  \frac{3n+2}{4} - |J| = |J^c| - \frac{n+2}{4}.
 \end{displaymath}
Using the assumption that $w \geq \frac{n+2}{4}$, we obtain 
 \begin{displaymath}
  \frac{n+2}{4} \leq w \leq |J^c| - \frac{n+2}{4} \leq \ell + |J^c| - \frac{n+2}{4}.
 \end{displaymath}
 So $F_{w,L\cup J^c} \in \oldF_n$ by definition.
\end{proof}

\begin{cor} \label{corollary:generation_up_top}
 Let $U \subseteq W$ be any $G$-stable open subset of $W$. The smallest
full triangulated subcategory of $\mathsf{D}^\mathsf{b}_{G}(U) \cong \mathsf{D^b}([U/G])$ containing the line bundles in $\oldF_n$ also contains, for each $J \subseteq \{0,\ldots,n\}$ with $|J| \leq \frac{n}{2}$ and each integer $w$ in the interval $\left[ \frac{3n+4}{4} - |J^c|, |J^c| - \frac{n+2}{4} \right]$, the objects $\mathcal O_{U \cap W^+_{\lambda_J}}(w(E-H))$ when $w \leq \frac n4$ and $\mathcal O_{U \cap W^+_{\lambda_J}}(w(E-H) - \sum_{i \in J^c} E_i)$ when $w \geq \frac{(n+2)}{4}$.
\end{cor}

\begin{proof}
 The Koszul complex $K(J)$ is quasi-isomorphic to $\mathcal
O_{W^+_{\lambda_J}}$. Lemma~\ref{lemma:gen_koszul} shows that $K(J)
\otimes \mathcal O(w(E-H))$ is a complex consisting of objects of
$\oldF_n$ when $w \leq \frac n4$ and $K(J) \otimes \mathcal
O(w(E-H) - \sum_{i \in J^c} E_i)$ is a complex consisting of objects of
$\oldF_n$. Thus, we can generate $\mathcal
O_{W^+_{\lambda_J}}(w(E-H))$ for $w \leq \frac n4$ and $\mathcal
O_{W^+_{\lambda_J}}(w(E-H) - \sum_{i \in J^c} E_i)$ for $w \geq
\frac{n+2}{4}$ in $\mathsf{D}^\mathsf{b}_{G}(W) \cong \mathsf{D^b}([W/G])$.
Since restriction to $U$ is exact, the analogous statement on $U$ holds. 
\end{proof}

Let us set up some notation to handle the full run of the MMP.

\begin{notation}\label{notation:full_run}
For any set $S$, let $P(\ell, S) = \{ J \subseteq S \mid |J| = \ell\}$.
If $S = \{0, \ldots, m\}$, let $P(\ell, S)$ be denoted $P(\ell, m)$.  We
view $P(\ell, m)$ as an ordered set using the lexicographic order and
denote its elements by $J_1, \ldots, J_{\binom{m}{\ell}}$.
This induces a total ordering on the full power set where the minimal
element is $\varnothing$.
For each $J$, we let $ \Sigma _{J}^+$ and $ \Sigma_{J}^-$
be the following chambers in the GIT fan: 
$$\Sigma_{J}^+ = \{ \chi \in \Sigma_{\text{GKZ}} \mid \widehat{\lambda}
_{J'}(\chi) > 0 \text{ for } J' \leq J  \text{ and }  \widehat{\lambda}
_{J'}(\chi) <  0 \text{ for } J' > J \}$$ 
$$ \Sigma_{J}^- = \{ \chi \in  \Sigma_{\text{GKZ}} \mid
\widehat{\lambda}_{J^\prime}(\chi) > 0 \text{ for } J^\prime < J \text{
and } \widehat{\lambda}_{J^\prime}(\chi) < 0 \text{ for } J^\prime \geq J\}.$$ 
We let $X_{J}^\pm$ be a GIT quotient for a linearization in
$\Sigma_{J}^\pm$, and $X_{J}^0$ the GIT quotient corresponding to the
generic linearization in the wall for $J$ \cite[Def. 14.3.13]{CLS}. The
sequence of birational maps given by crossing walls according to this
ordering begins at $V_n$ and terminates at $\bbP^n$.
\end{notation}

\begin{rem}
As observed by the referee, the run of the MMP described above can be
broken up into more comprensible steps.
Let $Y_n^{(j)}$ denote the stack corresponding to all subsets
$S$ of cardinality $\le j$.
Thus $Y_n^{(0)}$ corresponds to $\bbP^n$,
$Y_n^{(1)}$ corresponds to $Y_n$,
$Y_n^{(2)}$ corresponds to blowing up the strict transforms of the 
lines passing through the blown-up points, \dots,
and $Y_n^{(n/2)}$ corresponds to $V_n$.
Our wall-crossings can be reorganized into $(S_{n+1} \times C_2)$-equivariant
birational maps
\[
V_n = Y_n^{(0)} \dasharrow Y_n^{(1)} \dasharrow \cdots \dasharrow
Y_n^{(n/2)} = V_n.
\]
It is an interesting question to ask whether approprate subsets of
$F_{c,J}$ are still $(S_{n+1} \times C_2)$-stable exceptional collections
for the intervening stacks.
\end{rem}

As described in Section~\ref{section:window_prelim}, passing through the wall corresponding to $J$
yields a diagram where we replace the subscripts $\lambda_J$ with $J$
for brevity:
\begin{center}
 \begin{tikzpicture}
  \node at (-1.5,2) (x-) {$X_{J}^-$};
  \node at (1.5,2) (x+) {$X_{J}^+$};
  \node at (-3,0.75) (z-) {$Z_{J}^-$};
  \node at (3,0.75) (z+) {$Z_{J}^+$};
  \node at (0,0.5) (x0) {$X_{J}^0$};
  \node at (0,-1.5) (z0) {$Z_{J}^0$};
  \draw[->] (z-) -- node[above left] {$i^-$} (x-);
  \draw[->] (z+) -- node[above right] {$i^+$} (x+);
  \draw[->] (x-) -- node[below left] {$j^-$} (x0);
  \draw[->] (x+) -- node[below right] {$j^+$} (x0);
  \draw[->] (z-) -- node[below left] {$\pi^-$} (z0);
  \draw[->] (z+) -- node[below right] {$\pi^+$} (z0);
  \draw[->] (z0) -- node[left] {$i^0$} (x0);
  \draw[dashed,<->] (x-) -- (x+);
 \end{tikzpicture}
\end{center}

\begin{lem} \label{lem:unstableloci}
 We have isomorphisms
 \begin{gather*}
Z_{J}^+ \cong \bbP^{|J^c|-1} \\
Z_{J}^- \cong \bbP^{|J|-1} \\
Z_{J}^{0,\operatorname{rig}} \cong \operatorname{Spec} k
 \end{gather*}
where we recall the convention that $\bbP^0 \cong \operatorname{Spec} k$
and $\bbP^{-1} \cong \varnothing$.
 Moreover, these induce isomorphisms of sheaves $\O(D)|_{Z_J^\pm} \cong
\O_{Z_{J}^\pm}(\pm \widehat{\lambda}_J(D))$ for any divisor $D$ on $X_J^{\pm}$.
\end{lem}

\begin{proof}
We handle the $(+)$ claims. The statements for the $(-)$ side are proven completely analogously. 

On $W$, the ideal $(y_j \mid j \in J)$ defines the contracting locus $W_J^+$, so that functions on $W_J^+$ are given by
$k[x_0,\ldots,x_n] \otimes_k k[y_i \mid i \not \in J]$.
Assume that $y_i = 0$ for some point $p \in W^+_J$ and some $\ell \not \in J$.
Then $\lambda_{J \cup \{\ell \}}$ destabilizes $p$: $\widehat{\lambda}_{J \cup \{\ell \}}$ is negative on this chamber and $p$ lies in the contracting locus of $\lambda_{J \cup \{\ell\}}$, since $\lambda_{J \cup \{\ell\}}$ has positive weights on $k[x_0,\ldots,x_n] \otimes_k k[y_i \mid j \not \in J \cup \{\ell\}]$.

Assume that $x_j = 0$ for $j \in J$ for some point $p \in W^+_J$. Then $\lambda_J \lambda_{ J \setminus \{j\}}^{-1}$ destabilizes $p$. The weights $x_{\ell}$ for $\ell \not = j$ and $y_i$ for $i \not \in J$ are non-negative. The chamber $\Sigma_J^+$ lies in the positive half-spaces for both $\widehat{\lambda}_J$ and $\widehat{\lambda}_{J \setminus \{j\}}$. But $\widehat{\lambda}_J = 0$ intersects the closure of $\Sigma^+_J$. Thus, $\Sigma_J^+$ lies in the negative half-space associated to $\lambda_J \lambda_{ J \setminus \{j\}}^{-1}$. 

Additionally, $\lambda^{-1}_J$ destabilizes any point with all $x_i = 0$ for all $i \not \in J$. 

We have determined that we have a $G$-equivariant open immersion
\begin{displaymath}
 W^+_J \cap W^{\operatorname{ss}}(\chi^+) \subseteq \left( \mathbb{A}^{|J^c|} \setminus \{0\} \right) \times \mathbb{G}_m^{n+1}.
\end{displaymath}
The Hilbert-Mumford numerical criterion says that $W^{\operatorname{us}}(\chi^+)$ is the union of the contracting loci for one-parameter subgroups $\lambda$ with $\widehat{\lambda}(\chi^+) < 0$, i.e.,
\begin{displaymath}
 W^{\operatorname{us}}(\chi^+) = \bigcup_{\widehat{\lambda}(\chi^+) < 0} W^+_{\lambda}.
\end{displaymath}
We recall the subgroup $H$ from the proof of
Corollary~\ref{corollary:generation_up_top}.
Since the subgroup generated by $\delta_i$ for all $i$ acts by
multiplication on the torus factor, there is no contracting locus for
one-parameter subgroups.
Thus, we may pass directly to the invariant theory quotient by $H$
and then subsequently consider the GIT quotient by $G/H$. 

Taking the quotient by $H$ yields the subring $k[x_i y_i^{-1} \mid i \not \in J]$.  Note that $\lambda_J$ has weight $1$ on each $x_i y_i^{-1}$ for $i \not \in J$ and induces an isomorphism $G/H \cong \gm$. Thus, the above immersion is an equality and $Z_J^+ \cong \mathbb{P}^{|J^c|-1}$, where $\lambda_J$ induces the standard action on $\mathbb{A}^{|J^c|}$. 

Turning to the fixed locus, one can argue as above to conclude that 
\begin{displaymath}
 W^0_J \cap W^{\text{ss}}(\chi^0) \cong \gm^{n+1}
\end{displaymath}
with trivial $\lambda_J$-action. So we have $Z_J^0 \cong B\gm$.
Thus, $Z_J^{0,\operatorname{rig}} = \operatorname{Spec} k$.
\end{proof}

\begin{notation}
 Following the identification in Lemma~\ref{lem:unstableloci}, we will write $\mathcal O_{Z_J^+}(a)$ for the sheaf corresponding to $\mathcal O(a)$ on $\mathbb{P}^{|J^c|-1}$. 
\end{notation}

Let us now apply the framework introduced in Section~\ref{section:window_prelim} to our situation. For each $J \subseteq \{0,1,\ldots,n\}$ we can apply Theorem~\ref{thm:fundwin2} to the wall crossing at $\lambda_J$. 

\begin{prop} \label{prop:VnSOD}
 Let $J \subseteq \{0,1,\ldots,n\}$ with $|J| \leq \frac n2$. For any $d \in \Z$, there is a semi-orthogonal decomposition 
 \begin{displaymath}
  \mathsf{D^b}(X_{J}^+) = \left\langle \O_{Z_{J}^+}(d-|J^c|),\ldots,  \O_{Z_{J}^+}(d-1-|J|),\mathsf{D^b}(X_{J}^-)
\right\rangle,
 \end{displaymath}
 and if $\widehat{\lambda}_J(\chi) \in [d-|J^c|,d-1] = I_{d,J}^+$,
then $\Psi_d(\O_{X_{J}^+}(\chi)) = \O_{X_{J}^-}(\chi)$.
\end{prop}

\begin{proof}
This follows from Theorem~\ref{thm:fundwin2} and Lemma~\ref{lem:through_the_wall} as soon as we identify $\mathsf C_{\lambda}(i)$ with $\langle \O_{\bbP^{|J^c|-1}}(i) \rangle $. But Lemmas~\ref{lem:wall contributions} and~\ref{lem:unstableloci} give this identification.
\end{proof}

We now turn to generating the wall contributions. 

\begin{lem} \label{lem:d0}
 The values of $\widehat{\lambda}_J$ on $\oldF_n$ lie in the interval 
 \begin{displaymath}
  \left[\left\lceil \frac{n+2}{4} \right\rceil - |J^c| , \left\lceil \frac{n+2}{4} \right\rceil - 1\right] = I_{\lceil (n+2)/4 \rceil,J}^+. 
 \end{displaymath}
\end{lem}

\begin{proof}
We first note that
\begin{displaymath}
 \frac{n+2}{4} - |J^c| = |J| - \frac{3n+2}{4}.
\end{displaymath}

Recall that for $c \in \Z$ and $L \subseteq \{0,\ldots,n\}$, we have
$F_{c, L} = c(E - H) - \sum _{j \in L} E_j$.  We use the fact that
$\widehat{\lambda}_J(F_{c, L}) = |L \cap J|-c$. We check the above claim using the defining equations of $\oldF_n$. 

Let $(c, L)$ satisfy $|L| - \frac n4 \leq c \leq \frac n4$ and let $|J| \leq \frac n2$.
Clearly, $|L|-c \geq |L \cap J|-c \geq -c$.
Using our defining equation, we have $\frac n4 \geq |L|-c$. Combining these inequalities yields
$\frac n4 \geq |L|-c \geq |L \cap J|-c$, giving the claimed upper bound
on $|L \cap J|-c$.

To get the lower bound, we use $|L \cap J|-c \geq -c $ and $c \leq \frac n4$,
so that $|L \cap J|-c \geq - \frac {n}{4}$.
But $\frac n4 < \frac n4 + \frac 12 = \frac{n+2}{4} = \frac{3n-2n+2}{4}
= \frac{3n+2}4 - \frac{2n} 4$. Since $|J| \leq \frac n2$, we have $$| L
\cap J|-c \geq - \frac n4 > |J| - \frac{3n+2}{4}.$$

We now show the claim holds when the second defining equation of $\oldF_{n}$ is satisfied. Assume $(c, L)$ satisfies $\frac{n+2}{4} \leq c \leq |L| - \frac{n+2}{4}$. We have 
\begin{align*}
 c - |L \cap J| & \geq c - |J| \\
& \geq \frac{n+2}{4} - |J| \\
& \geq \frac{n+2}{4} - \frac n2 \\
& \geq -\frac n4.
\end{align*}

For the claimed lower bound, we have 
\begin{align*}
 |L| - |L \cap J| & = |L \cap J^c| \\
 & \leq |J^c| \\
 & = n+1 - |J| \\
 & = \frac{3n+2}{4} + \frac{n+2}{4} - |J|.
\end{align*}
Subtracting $\frac{n+2}{4}$ from the first and last terms gives 
 \[
 c-|L \cap J| \leq |L| - \frac{n+2}{4} - |L \cap J| \leq \frac{3n+2}{4} - |J|.
 \]
 We have shown the weights lie in the set 
 \begin{displaymath}
   \left[ \frac{n+2}{4} - |J^c| ,  \frac{n}{4} \right] \cap \Z = 
  \left[\left \lceil \frac{n+2}{4} \right \rceil - |J^c| , \left \lfloor \frac{n}{4} \right \rfloor \right] \cap \Z.
 \end{displaymath}
 For even $n$, we have
$\left\lceil \frac{n+2}{4} \right\rceil - 1 = \left \lfloor \frac{n}{4}
\right \rfloor$.  Thus
 \begin{displaymath}
  \left[\left \lceil \frac{n+2}{4} \right \rceil - |J^c| , \left \lfloor \frac{n}{4} \right \rfloor \right] = \left[\left\lceil \frac{n+2}{4} \right\rceil - |J^c| , \left\lceil \frac{n+2}{4} \right\rceil - 1\right]. 
 \end{displaymath}
\end{proof}

\begin{lem} \label{lem:genpieces}
 The collection $\oldF_n$, viewed as line bundles on $X_{J}^+$, generates the sheaves
\begin{displaymath}
 \O_{Z_{J}^+}\left( \left\lceil \frac{n+2}{4} \right\rceil - |J^c| \right),\ldots,\O_{Z_{J}^+}\left( \left\lceil \frac{n+2}{4} \right\rceil -1- |J| \right).
\end{displaymath}
\end{lem}

\begin{proof}
Assume that $w$ lies in 
\begin{equation}\label{eqn:interval}
 \left[ \frac{3n+4}{4} - |J^c|, |J^c| - \frac{n+2}{4} \right] \cap \mathbb{Z}.
\end{equation}
Let $U$ be a $G$-stable open subset of $W$. From Corollary~\ref{corollary:generation_up_top}, if $w \leq \frac n4$, then the collection $\oldF_n$ generates $\O_{U \cap W_{J}^+}(w(E -H))$. If $w \geq \frac{(n+2)}{4}$, then $\oldF_n$ generates $\O_{U \cap W_{J}^+}(w(E-H) - \sum_{i\in J^c} E_i)$. We take $U = W^{\operatorname{ss}}(\chi^+)$ and recall that $Z_J^+$ is $\left[ U \cap W_J^+ / G \right]$. 

Thus, $\oldF_n$ generates $\mathcal O_{Z_J^+}(w(E-H))$ for $w \leq \frac n4$ and $\mathcal O_{Z_J^+}(w(E-H) - \sum_{i\in J^c} E_i)$ for $w \geq \frac{(n+2)}{4}$. The weights with respect to $\lambda_J$ are 
\begin{align*}
 \widehat{\lambda}_J\left(w(E-H)\right) & = - w + |J \cap \varnothing| = -w \\
 \widehat{\lambda}_J\left(w(E-H) - \sum_{i \in J^c} E_i\right) & = -w + |J \cap J^c| = -w. 
\end{align*}
Applying Lemma~\ref{lem:unstableloci} and using the $\lambda_J$-weight computations, when $w$ lies in the interval given in \eqref{eqn:interval},
the collection $\oldF_n$ will generate $\mathcal O_{Z_J^+}(a)$ for
\begin{displaymath}
 a \in \left[ \frac{n+2}{4} - |J^c|, |J^c| - \frac{3n+4}{4} \right] \cap \mathbb{Z},
\end{displaymath}
which is the same interval as in the statement of the lemma.
\end{proof}

Finally, we can easily handle the generation result for $\oldF_n$.

\begin{thm} 
The collection $\oldF_n$ generates the category $\mathsf{D^b}(V_n)$.
\end{thm}

\begin{proof}
Set 
\begin{displaymath}
 d := \left \lceil \frac{n+2}{4} \right \rceil. 
\end{displaymath}
Using Proposition~\ref{prop:VnSOD}, we have a semi-orthogonal decomposition
\begin{displaymath}
  \mathsf{D^b}(X_{J}^+) = \left\langle \O_{Z_{J}^+}\left(\left \lceil \frac{n+2}{4} \right \rceil-|J^c|\right), \ldots,  \O_{Z_{J}^+}\left( \left \lceil \frac{n+2}{4} \right \rceil - 1 - |J| \right),\mathsf{D^b}(X_{J}^-) \right\rangle.
\end{displaymath}
By Lemma~\ref{lem:d0}, the collection $\oldF_n$, viewed as line bundles, generates the components
\begin{displaymath}
 \O_{Z_{J}^+}\left(\left \lceil \frac{n+2}{4} \right \rceil-|J^c|\right), \ldots,  \O_{Z_{J}^+}\left( \left \lceil \frac{n+2}{4} \right \rceil - 1 - |J| \right).
\end{displaymath}

To show that $\oldF_n$ generates $\mathsf{D^b}(V_n)$, we work via
(downward) induction on the lexicographic ordering given above on $J
\subseteq \{0,\ldots,n\}$ for all $|J| \leq \frac n2$.
Using Lemma~\ref{lem:gen} and the semi-orthogonal decomposition above, we see that $\oldF_n$ generates $\mathsf{D^b}(X_{J}^+)$ if $\Psi_{\lceil (n+2)/4 \rceil}(\oldF_n)$ generates $\mathsf{D^b}(X_{J}^-)$. 

Using the second statement of Proposition~\ref{prop:VnSOD} and the weights of $\oldF_n$ computed in Lemma~\ref{lem:d0}, we see that $\Psi_{\lceil (n+2)/4 \rceil}(\oldF_n) = \oldF_n$ (recall we are identifying these elements with their corresponding line bundles). Thus, we reduce to the base case of the induction: $X^-_{\varnothing} = \varnothing$. Since $\mathsf{D^b}(X^-_\varnothing) = 0$, the statement here is trivial.
(One can alternatively start the induction with non-empty $J$ since
$X^+_\varnothing \cong \mathbb{P}^n$
where  $\mathsf{D^b}(X^+_\varnothing)$ is generated by Beilinson's collection.)
\end{proof}


\bibliographystyle{alpha}
\bibliography{ExcCollToric}

\end{document}